\documentclass[hidelinks,onefignum,onetabnum,onealgnum]{siamart251216}

\usepackage{amsfonts}
\usepackage{amssymb}
\usepackage{graphicx}
\usepackage{epstopdf}

\newsiamremark{remark}{Remark}
\newsiamremark{hypothesis}{Hypothesis}
\crefname{hypothesis}{Hypothesis}{Hypotheses}
\newsiamthm{claim}{Claim}
\newsiamremark{assumption}{Assumption}
\crefname{assumption}{Assumption}{Assumptions}
\crefname{appendix}{Appendix}{Appendices}
\Crefname{appendix}{Appendix}{Appendices}

\usepackage[T1]{fontenc}
\usepackage[utf8]{inputenc}
\usepackage{mathtools}
\usepackage{bm}
\usepackage{booktabs}
\usepackage{multirow}
\usepackage{array}
\usepackage{stmaryrd}
\usepackage[labelsep=period, textfont=it, labelfont={normalfont,sc}]{caption}
\usepackage{subcaption}
\usepackage{enumitem}
\usepackage{tikz}
\usetikzlibrary{shapes.geometric,arrows.meta,positioning,calc}
\usepackage{tikz-cd}
\usepackage{algpseudocode}

\newcommand{\R}{\mathbb{R}}

\newcommand{\bu}{\bm{u}}
\newcommand{\bv}{\bm{v}}

\newcommand{\Net}{\operatorname{Net}}
\newcommand{\MSE}{\operatorname{MSE}}
\newcommand{\Res}{\mathcal{R}}
\newcommand{\EG}{\mathcal{E}_G}
\newcommand{\ET}{\mathcal{E}_T}
\newcommand{\norm}[2]{\left\|#1\right\|_{#2}}

\graphicspath{{figure/}{fig/}}

\headers{Jin--Xin relaxation gradual convergence for PINNs}{E. Lorin, Y. Tang, X. Yang, and Y. Zhu}

\title{A Jin--Xin Relaxation Gradual Convergence Method for
  Conservation-Law PINNs%
  \thanks{Submitted on \today.
    \funding{X.Y. was partially supported by the ONR grant under \#N00014-24-1-2432.}}}

\author{
  Emmanuel Lorin\thanks{School of Mathematics and Statistics, Carleton University, Ottawa, ON K1S 5B6, Canada; and Centre de Recherches
    Math\'ematiques, Universit\'e de Montr\'eal, Montr\'eal, QC H3T 1J4,
    Canada. \texttt{elorin@math.carleton.ca}}
  \and
  Yuchong Tang\thanks{Qiuzhen College, Tsinghua University, Beijing 100084, China.
    \texttt{tyc22@mails.tsinghua.edu.cn}}
  \and
  Xu Yang\thanks{Department of Mathematics, University of California,
    Santa Barbara, CA 93106, USA. \texttt{xy6@ucsb.edu}}
  \and
  Yi Zhu\thanks{Yau Mathematical Sciences Center, Tsinghua University, Beijing 100084, China; and
Yanqi Lake Beijing Institute of Mathematical Sciences and Applications, Beijing 101408, China
  (\email{yizhu@tsinghua.edu.cn}).}
}

\usepackage{amsopn}

\usepackage{placeins}
\usepackage{enumitem}

\ifpdf
\hypersetup{
  pdftitle={A Jin-Xin Relaxation Gradual Convergence Method for PINNs},
  pdfauthor={E. Lorin, Y. Tang, X. Yang, and Y. Zhu}
}
\fi

\begin{document}
\maketitle

\begin{abstract}
The Jin--Xin relaxation of a nonlinear hyperbolic conservation law introduces a relaxation parameter that controls the width of the internal layer resolving a shock; the discontinuity of the limiting conservation law emerges only in the singular limit as this width vanishes.  Physics-informed neural networks (PINNs) use smooth network approximations and are therefore not well suited to this limit, while relaxation PINNs with a fixed parameter resolve only a single scale and cannot follow the multiscale transition toward the limiting solution.  We propose the Jin--Xin relaxation gradual convergence method (JXRGCM), which treats the relaxation parameter as a continuation variable, annealing it to zero along a schedule and warm-starting each stage from the previous one, so that the approximation follows the relaxation profile through progressively sharper scales.  Under the sub-characteristic condition we establish a stability estimate whose constant is independent of the relaxation parameter; combined with the relaxation limit, it yields for scalar conservation laws an $L^2$ convergence rate of $\mathcal{O}(\varepsilon^{1/4})$ toward the entropy solution. Numerical experiments on the Burgers equation, the shallow-water dam-break problem, and the Sod shock tube show that JXRGCM improves shock and rarefaction resolution compared with fixed-parameter relaxation PINNs and other physics-informed approaches.
\end{abstract}

% REQUIRED

% REQUIRED

%======================================================================
\section{Introduction}\label{sec:intro}
We consider the one-dimensional system of conservation laws
\begin{equation}\label{eq:cons-law}
\partial_t \mathbf{u} + \partial_x \mathbf{f}\bigl(\mathbf{u}\bigr) = 0, 
\end{equation}
subject to the initial condition
\begin{equation}\label{eq:ic}
\mathbf{u}(0,x) = \mathbf{u}_0(x), \quad x \in \mathbb{R},
\end{equation}
where \(\mathbf{u}:\mathbb{R}^{+}\times\mathbb{R}\to\mathbb{R}^{n}\) is the vector of conserved variables, \(\mathbf{f}: \mathbb{R}^{n}\to \mathbb{R}^{n}\) is the flux, and \(\mathbf{u}_0(x)\) denotes the prescribed initial data. Solutions of \cref{eq:cons-law} generically develop discontinuities (shocks and contact discontinuities) in finite time even from smooth initial data, with the admissible weak solution selected by the Rankine--Hugoniot and entropy conditions \cite{LeVeque2002,Toro2009}.  Such discontinuities are central to aerodynamics \cite{Anderson1989}, transport in porous media \cite{FuksTchelepi2020}, molecular diffusion \cite{Crank1979}, and many other settings.  Classical shock-capturing discretizations resolve them through total-variation-diminishing schemes \cite{Harten1983}, weighted essentially non-oscillatory reconstructions \cite{LiuOsherChan1994}, adaptive mesh refinement \cite{BergerColella1989}, or entropy-stable discontinuous Galerkin methods \cite{ChenShu2020}, but balancing accuracy, robustness, and efficiency for complex discontinuous problems remains demanding.

Physics-informed neural networks (PINNs) approximate the solution by a smooth network trained to minimize the residual of the governing equation together with data-fitting terms  \cite{Raissi2019,Lagaris1998,Hornik1989}.  They are mesh-free, adapt naturally to high dimension, and integrate heterogeneous data, but perform poorly on hyperbolic conservation laws with discontinuities for three intertwined reasons.  First, the continuous, high-order-differentiable network conflicts with the discontinuous weak solution; the spectral bias of networks toward low frequencies makes sharp fronts hard to fit and induces Gibbs-type oscillations \cite{GhoreishiNaderan2026}.  Second, the strong-form $L^2$ residual is ill-defined across a shock, where derivatives behave like a Dirac mass, so the optimizer over-dissipates the front and the interior, initial, and boundary losses conflict near the discontinuity \cite{Krishnapriyan2021}.  Third, the resulting non-smooth landscape degrades gradient-based training \cite{WangTengPerdikaris2021}.  The PINN generalization error is, moreover, controlled by the stability of the underlying PDE, i.e.\ the sensitivity of the solution to residual perturbations \cite{DeRyckMishra2024}.

\paragraph{Related work}  Following the survey of \cite{Abbasi2025}, remedies fall into three groups.  {Physical modifications} change the governing equation: artificial viscosity \cite{VonNeumannRichtmyer1950,GhoreishiNaderan2026}, entropy-condition regularization \cite{Patel2022} and coupled-integral conservation \cite{WangYang2024}, weak-form residuals (wPINNs) \cite{wPINN2024}, and relaxation of the conservation law into an auxiliary-flux system, as in the relaxation neural network (RelaxNN) \cite{ZhouMa2024} and approximate-Riemann formulations \cite{UrbanPons2025}.  {Loss and training modifications} reweight or resample without changing the equation, e.g.\ gradient-annihilation weighting and residual-adaptive sampling \cite{Wu2023}.  {Architectural modifications} change the network, e.g.\ transformer backbones \cite{Zhao2023} and space-time domain decomposition \cite{JagtapKarniadakis2020,LorinNovruzi}.  The relaxation approach is attractive because it converts the stiff strong-form constraint into a benign $L^2$ coupling: writing $\partial_t u+\partial_x v=0$, $v=f(u)$, the spatial derivative acts on the smoother auxiliary flux $v_\theta$ rather than on $f(u_\theta)$, and the constitutive relation becomes a local $L^2$ penalty \cite{ZhouMa2024}.  However, RelaxNN and related methods fix the relaxation parameter $\varepsilon$ at a single value: too large leaves the discontinuity only approximately resolved, while too small reinstates the stiffness, so the smooth network is again asked to fit a profile it cannot represent and spurious shock-like structures reappear.

The difficulty is better understood by asking what scale the network is being asked to resolve.  The relaxation parameter $\varepsilon$ is not an arbitrary regularization strength: by the Chapman--Enskog expansion \cref{eq:chapman-enskog} it sets the effective diffusion, and hence the width $\mathcal{O}(\varepsilon)$ of the internal layer through which the relaxed solution passes from one side of a shock to the other.  A network of fixed capacity can represent that layer while it is wide and cannot once it is narrow, so a fixed-$\varepsilon$ method must either stop at a scale where the discontinuity is still smeared or attempt, in one step, a profile it has no means to express.  The conservation law itself sits at the end of this family as the singular limit $\varepsilon\to0$, where the layer has no width at all.

The idea of this work is to treat $\varepsilon$ as a continuation variable rather than a fixed constant, and thereby to approach the singular limit through the family of scales instead of jumping to it.  The {Jin--Xin relaxation gradual convergence method} (JXRGCM) trains a PINN on the Jin--Xin relaxation system and anneals $\varepsilon\to0$ along a schedule, warm-starting each stage from the parameters of the previous one, so that the trained solution follows the relaxation profile down through the scales and tracks the conservation-law solution as the limit sharpens (\cref{sec:method}).

The analysis rests on an $\varepsilon$-uniform stability estimate controlling the total error by the generalization error (\cref{thm:total-uniform}), obtained from a symmetrizer built on the sub-characteristic condition. By contrast, the direct estimate in \cref{thm:total-nonuniform} has a constant that deteriorates as $\varepsilon\to0$. This uniformity is what makes the continuation meaningful, since the approximation error does not deteriorate as the schedule descends through the scales. Combined with the relaxation limit (\cref{lem:relax-limit}), the estimate yields, for scalar conservation laws, an $L^2$ convergence rate of $\mathcal{O}(\varepsilon^{1/4})$ toward the entropy solution (\cref{thm:overall}); a quadrature estimate relates the training and generalization errors (\cref{thm:quadrature}). In \cref{sec:numerics}, the method is compared against the fixed-$\varepsilon$ RelaxNN \cite{ZhouMa2024}, an artificial-viscosity PINN \cite{VonNeumannRichtmyer1950}, the gradient-annihilating PINN (GA-PINN) \cite{ferrer2024gradient}, and the coupled-integral PINN (CI-PINN) \cite{WangYang2024}.

The remainder of the paper is organized accordingly: Section~\ref{sec:method} presents the relaxation formulation and continuation algorithm, Section~\ref{sec:theory} develops the analysis, Section~\ref{sec:numerics} reports the numerical experiments, and we conclude in Section~\ref{sec:conclusion}. \Cref{app:CE} contains the Chapman--Enskog derivation of the sub-characteristic condition.

%======================================================================
\section{The Jin--Xin relaxation gradual convergence method}\label{sec:method}

This section introduces the Jin--Xin relaxation formulation and the gradual continuation strategy underlying JXRGCM. The relaxation parameter $\varepsilon$ determines the scale of the internal layer, so decreasing $\varepsilon$ produces progressively sharper approximations of the limiting entropy solution. Rather than training directly at a very small value of $\varepsilon$, JXRGCM follows this hierarchy of scales by solving a sequence of relaxed problems and warm-starting each stage from the solution obtained at the previous one. We first describe the relaxed system and its PINN formulation, and then present the continuation algorithm.

\subsection{Relaxation formulation}\label{subsec:relaxation}
Introducing a relaxation variable $\mathbf{v}(t,x):\mathbb{R}^{+}\times\mathbb{R}\to\mathbb{R}^{n}$, we turn \cref{eq:cons-law} into the form of the Jin--Xin relaxation system \cite{JinXin1995}
\begin{equation}\label{eq:jinxin}
  \left\{
  \begin{aligned}
    &\partial_t \mathbf{u} + \partial_x \mathbf{v} = 0,\\
    &\varepsilon\bigl(\partial_t \mathbf{v} + \mathbf{A}^2 \partial_x \mathbf{u}\bigr)
      = \mathbf{f}(\mathbf{u}) - \mathbf{v},
  \end{aligned}
  \right.
\end{equation}
subject to the local equilibrium initial conditions
\begin{equation}\label{eq:jinxin-ic}
  \mathbf{u}(0,x) = \mathbf{u}_0(x) \quad \text{and} \quad \mathbf{v}(0,x) = \mathbf{f}\bigl(\mathbf{u}_0(x)\bigr),
\end{equation}
where $\varepsilon>0$ is the relaxation parameter and $\mathbf{A} = \operatorname{diag}\{a_1, a_2, \dots, a_n\}$ is a constant  positive diagonal matrix.
For fixed $\varepsilon$, the system is semilinear and hyperbolic; as $\varepsilon\to0$ the second equation enforces $\mathbf{v} \to \mathbf{f}(\mathbf{u})$ and \cref{eq:jinxin} formally reduces to \cref{eq:cons-law}.  A Chapman--Enskog expansion (\cref{app:CE}) yields the first-order modified equation
\begin{equation}\label{eq:chapman-enskog}
  \partial_t \mathbf{u} + \partial_x \mathbf{f}(\mathbf{u})
  = \varepsilon\,\partial_x\!\Bigl(\bigl(\mathbf{A}^2 - \mathbf{Df}(\mathbf{u})^2\bigr)\partial_x \mathbf{u}\Bigr)
    + \mathcal{O}(\varepsilon^2),
\end{equation}
so the leading correction is dissipative, and the limit consistent and stable, precisely under the \emph{sub-characteristic condition} \cite{ChenLevermoreLiu1994}: considering $\mathbf{A}^2 - \mathbf{Df}(\mathbf{u})$ with eigenvalues $\lambda_i(\mathbf{u})$,
\begin{equation}\label{eq:subchar}
  \lambda_i(\mathbf{u}) > 0 
\end{equation}
over the relevant range of states.

\subsection{PINN approximation and residuals}\label{subsec:pinn}
We approximate the conserved variable and relaxation flux by a single network with parameters $\theta$,
\begin{equation}\label{eq:net}
  (\mathbf{u}_{\theta},\mathbf{v}_{\theta}) = \Net(t,x;\theta).
\end{equation}
By automatic differentiation, the interior residuals of \cref{eq:jinxin} are
\begin{equation}\label{eq:residuals}
\begin{cases}
  \Res_{\mathbf{u}}[\theta] = \partial_t \mathbf{u}_{\theta} + \partial_x \mathbf{v}_{\theta}, \\
  \Res_{\mathbf{v}}[\theta] = \varepsilon\bigl(\partial_t \mathbf{v}_{\theta}
                    + \mathbf{A}^2 \partial_x \mathbf{u}_{\theta}\bigr)
                    - \bigl(\mathbf{f}(\mathbf{u}_{\theta}) - \mathbf{v}_{\theta}\bigr). 
\end{cases}
\end{equation}
The training objective combines a data term and a residual term:
\begin{equation}\label{eq:loss}
  \mathcal{L}(\theta)
  = \lambda_{\mathrm{data}}\,\mathcal{L}_{\mathrm{data}}(\theta)
  + \lambda_{\mathrm{pde}}\,\mathcal{L}_{\mathrm{pde}}(\theta),
\end{equation}
\begin{align}
  \mathcal{L}_{\mathrm{data}}(\theta)
    &= \MSE\bigl(\mathbf{u}_{\theta}(0,x),\mathbf{u}_0(x)\bigr)
     + \MSE\bigl(\mathbf{v}_{\theta}(0,x),\mathbf{f}(\mathbf{u}_0(x))\bigr),
       \label{eq:loss-data}\\
  \mathcal{L}_{\mathrm{pde}}(\theta)
    &= \MSE\bigl(\Res_{\mathbf{u}}[\theta],0\bigr)
     + \MSE\bigl(\Res_{\mathbf{v}}[\theta],0\bigr),
       \label{eq:loss-pde}
\end{align}
with $\lambda_{\mathrm{data}},\lambda_{\mathrm{pde}}>0$.

\subsection{Gradual convergence in \texorpdfstring{$\varepsilon$}{epsilon}}
\label{subsec:continuation}
By \cref{eq:chapman-enskog} the relaxation parameter $\varepsilon$ sets the effective diffusion $\varepsilon(\mathbf{A}^2 - \mathbf{Df}(\mathbf{u})^2)$ and hence the width of the internal layer that resolves a shock, so annealing $\varepsilon$ is a continuation in \emph{scale} rather than a training heuristic: each stage asks the network to represent a profile one step narrower than the one it already represents, and the conservation law is recovered only in the singular limit, where the layer width vanishes.  The idea is summarized by \cref{fig:diagram}. Let $\Res_\varepsilon$ denote the relaxation system \cref{eq:jinxin}, $\Res_0$ the conservation law \cref{eq:cons-law}, and $\mathcal{F}_\varepsilon,\mathcal{F}_0$ their solutions. A direct PINN takes the left-then-bottom path and fails, since $\mathcal{F}_0$ is discontinuous.  JXRGCM follows the top-then-right path: it trains a PINN to drive the relaxed loss to zero (top edge), then anneals $\varepsilon\to0$ (right edge) while fine-tuning the network so that $\mathcal{F}_\varepsilon$ tracks $\mathcal{F}_0$.  The diagram expresses the intended limiting relation between the two paths rather than a literally commuting square, since the direct bottom path trains on a strong-form residual that is ill-defined at a discontinuity; the sense in which the top-then-right path reaches $\mathcal{F}_0$ is made precise in \cref{thm:overall}.

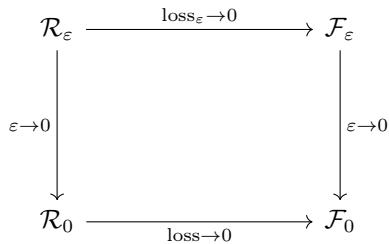
\begin{figure}[htbp]
  \centering
  \begin{tikzcd}[column sep=3cm, row sep=2cm]
    \Res_\varepsilon
      \arrow[r, "{\text{loss}_\varepsilon \to 0}"]
      \arrow[d, "\varepsilon \to 0"']
    & \mathcal{F}_\varepsilon
      \arrow[d, "\varepsilon \to 0"] \\
    \Res_0
      \arrow[r, "{\text{loss} \to 0}"']
    & \mathcal{F}_0
  \end{tikzcd}
  \caption{Gradual-convergence strategy.  Rows correspond to the relaxation
    system $\Res_\varepsilon$ and the conservation law $\Res_0$, columns to their
    solutions $\mathcal{F}_\varepsilon$ and $\mathcal{F}_0$.  Training on the
    relaxed system (top) and then annealing $\varepsilon\to0$ (right) reaches the
    entropy solution $\mathcal{F}_0$, whereas the direct fit (bottom) trains on a
    strong-form residual that is ill-defined at a discontinuity, which is where
    standard PINNs fail.}
  \label{fig:diagram}
\end{figure}

\subsection{Algorithm}\label{subsec:algorithm}
\Cref{alg:jxrgcm} states the method.  The outer loop decreases $\varepsilon$ geometrically by a factor $r\in(0,1)$; each stage trains on the relaxation system at the current $\varepsilon$ with a stage tolerance $L_{\min}$ that tightens as $\varepsilon$ shrinks, and parameters are carried over between stages (warm start).  A final stage sets $\varepsilon=0$ to sharpen the solution near the limit.

\begin{algorithm}[htbp]
\caption{Jin--Xin relaxation gradual convergence method (JXRGCM)}
\label{alg:jxrgcm}
\begin{algorithmic}[1]
  \Require initial parameter $\varepsilon_0$; decay rate $r\in(0,1)$; base
    learning rate $\eta_0$; learning-rate factor $\rho$; base loss level $L$;
    weights $\lambda_{\mathrm{data}},\lambda_{\mathrm{pde}}$;
    termination parameter $\varepsilon^\ast$; maximum epochs
    $\mathrm{max\_epochs}$; final loss threshold $L_{\min}^{\mathrm{final}}$
  \State Initialize parameters $\theta$; set
    $\mathrm{first\_stage}\gets\mathrm{true}$, $\varepsilon\gets\varepsilon_0$
  \While{$\varepsilon \ge \varepsilon^\ast$}
    \State $\eta \gets \eta_0$ if $\mathrm{first\_stage}$, else
      $\eta \gets \eta_0\,\rho$
    \State $L_{\min} \gets L\cdot \varepsilon/\varepsilon_0$
    \State Sample initial data $\{(0,x_0^i,\mathbf{u}_0(x_0^i))\}_{i=1}^N$ and collocation
      points $\{(t_j,x_j)\}_{j=1}^M$
    \For{$\mathrm{epoch}=1$ to $\mathrm{max\_epochs}$}
      \For{each mini-batch $B$}
        \State Evaluate $(\mathbf{u}_{\theta},\mathbf{v}_{\theta})$ at data and collocation points
        \State Compute $\mathcal{L}_{\mathrm{data}}$, $\mathcal{L}_{\mathrm{pde}}$
          from \cref{eq:loss-data,eq:loss-pde}, then
          $\mathcal{L}=\lambda_{\mathrm{data}}\mathcal{L}_{\mathrm{data}}
            +\lambda_{\mathrm{pde}}\mathcal{L}_{\mathrm{pde}}$
        \State Update $\theta$ by a gradient step with learning rate $\eta$
      \EndFor
      \If{$\mathcal{L} < L_{\min}$} \State \textbf{break} \EndIf
    \EndFor
    \State $\mathrm{first\_stage}\gets\mathrm{false}$;\quad
      $\varepsilon \gets r\,\varepsilon$
  \EndWhile
  \State \textbf{Final refinement:} set $\varepsilon\gets0$,
    $L_{\min}\gets L_{\min}^{\mathrm{final}}$, $\eta\gets\eta_0\,\rho^2$, and
    train as above
  \State \Return $\theta$
\end{algorithmic}
\end{algorithm}

%======================================================================
\section{Theoretical analysis}\label{sec:theory}
In this section, we consider the vector relaxation system in symmetric form
\begin{equation}\label{eq:jinxin-sym}
  \begin{cases}
    \partial_t \mathbf{u} + \partial_x \mathbf{v} = 0, \\
    \varepsilon\bigl(\partial_t \mathbf{v} + \mathbf{A}^2\partial_x \mathbf{u}\bigr) = \mathbf{f}(\mathbf{u}) - \mathbf{v},
  \end{cases}
\end{equation}
with $\mathbf{u},\mathbf{v}:\mathbb{R}^+\times\mathbb{R}\to\mathbb{R}^{n}$, $\mathbf{f}\in C^\infty(\mathbb{R}^n)$, and $\mathbf{A}\in\mathbb{R}^{n\times n}$ symmetric positive definite, equipped with the equilibrium initial data
\begin{equation}\label{eq:equilibrium-ic}
\begin{cases}
  \mathbf{u}(0,x) = \mathbf{u}_0(x), \\
  \mathbf{v}(0,x) = \mathbf{f}\bigl(\mathbf{u}_0(x)\bigr),
\end{cases}
\end{equation}
where $\mathbf{u}_0\in(H^s(\mathbb{R}))^n$.

\subsection{Regularity of the Relaxation System}\label{subsec:regularity}
We first present several properties of the relaxation system.

\begin{proposition}[Local, Spacetime Regularity, and Uniform Regularity]\label{prop:regularity-uniform}
Let $s \ge 1$, $\mathbf{u}_0 \in (H^s(\mathbb{R}))^n$, and consider system \cref{eq:jinxin-sym} equipped with the equilibrium initial data \cref{eq:equilibrium-ic}.

\begin{enumerate}[label=(\roman*)]
  \item Local existence \cite{LiuYong2001}:\label{thm:local-existence} 
  For each fixed $\varepsilon>0$, there exists $T^\varepsilon>0$ such that \cref{eq:jinxin-sym,eq:equilibrium-ic} has a unique solution $(\mathbf{u},\mathbf{v})\in(C([0,T^\varepsilon];H^s(\mathbb{R})))^n\times(C([0,T^\varepsilon];H^s(\mathbb{R})))^n$. Note that $T^\varepsilon$ need not be bounded below as $\varepsilon\to0$.

  \item Spacetime regularity:\label{thm:spacetime} 
  If $s\in\mathbb{N}$, then $(\mathbf{u},\mathbf{v})\in(H^s([0,T^\varepsilon]\times\mathbb{R}))^n\times(H^s([0,T^\varepsilon]\times\mathbb{R}))^n$.

  \item Uniform-in-$\varepsilon$ existence \cite{Yong1993}:\label{thm:uniform-existence} 
  If the sub-characteristic condition \cref{eq:subchar} holds and $s\ge 3/2$, then there exist a time $T^\ast>0$ independent of $\varepsilon$, and a constant $C>0$ independent of $\varepsilon$, such that
  \begin{equation}\label{eq:uniform-bound}
    \sup_{t\in[0,T^\ast]}\bigl(\|\mathbf{u}(t,\cdot)\|_{H^s(\mathbb{R})}^2 
    + \|\mathbf{v}(t,\cdot)\|_{H^s(\mathbb{R})}^2\bigr) \le C.
  \end{equation}
\end{enumerate}
\end{proposition}

\begin{remark}\label{rem:multi-d-regularity}
Regarding item~(iii) in Proposition~\ref{prop:regularity-uniform}, for $d$ spatial dimensions ($x \in \mathbb{R}^d$), the corresponding Sobolev regularity requirement for the uniform estimate \eqref{eq:uniform-bound} is $s > 1 + d/2$.
\end{remark}

\begin{proof}[Proof of Proposition~\ref{prop:regularity-uniform}(ii)]
We show the case $s=1$ and indicate the induction for $s \ge 2$. By \cref{thm:local-existence}, $\mathbf{u},\mathbf{v}\in C([0,T^\varepsilon];H^1(\mathbb{R})) \hookrightarrow C([0,T^\varepsilon];L^\infty(\mathbb{R}))$. Since $\mathbf{f}\in C^\infty$, we have $\mathbf{f}(\mathbf{u})-\mathbf{v}\in C([0,T^\varepsilon];L^2(\mathbb{R}))$. Using $\partial_t\mathbf{u}=-\partial_x\mathbf{v}$, we obtain
\begin{equation*}
  \begin{aligned}
    \|\mathbf{u}\|_{H^1([0,T^\varepsilon]\times\mathbb{R})}^2
    &= \int_0^{T^\varepsilon}\!\bigl(\|\mathbf{u}(t,\cdot)\|_{L^2(\mathbb{R})}^2
      +\|\partial_x\mathbf{v}(t,\cdot)\|_{L^2(\mathbb{R})}^2
      +\|\partial_x\mathbf{u}(t,\cdot)\|_{L^2(\mathbb{R})}^2\bigr)\,\mathrm{d}t \\
    &< \infty,
  \end{aligned}
\end{equation*}
and using $\partial_t\mathbf{v}=-\mathbf{A}^2\partial_x\mathbf{u}+\frac{1}{\varepsilon}(\mathbf{f}(\mathbf{u})-\mathbf{v})$,
\begin{equation*}
  \begin{aligned}
    \|\mathbf{v}\|_{H^1([0,T^\varepsilon]\times\mathbb{R})}^2
    &\le T^\varepsilon\!\sup_{t\in[0,T^\varepsilon]}\!\Bigl(
      \|\mathbf{v}(t,\cdot)\|_{L^2(\mathbb{R})}^2+\|\partial_x\mathbf{v}(t,\cdot)\|_{L^2(\mathbb{R})}^2
      +\|\mathbf{A}\|^2\|\partial_x\mathbf{u}(t,\cdot)\|_{L^2(\mathbb{R})}^2 \\
    &\qquad\qquad\qquad +\frac{1}{\varepsilon^2}\|\mathbf{f}(\mathbf{u}(t,\cdot))-\mathbf{v}(t,\cdot)\|_{L^2(\mathbb{R})}^2\Bigr) < \infty.
  \end{aligned}
\end{equation*}
For $s\ge2$, the spatial derivatives satisfy $\partial_x^p\mathbf{u},\partial_x^p\mathbf{v}\in C([0,T^\varepsilon];H^{s-p}(\mathbb{R}))$ for $p\le s$. Since $H^s(\mathbb{R})$ is an algebra, an induction on the mixed order $\alpha_1+\alpha_2$ converts each time derivative to spatial derivatives via \cref{eq:jinxin-sym} (e.g., $\partial_t\partial_t^{\alpha_1}\partial_x^{\alpha_2-1}\mathbf{u} =-\partial_x\partial_t^{\alpha_1}\partial_x^{\alpha_2-1}\mathbf{v}$, and similarly for $\mathbf{v}$ with the chain rule controlling $\partial^\alpha(\mathbf{f}(\mathbf{u})-\mathbf{v})$), which gives $\partial_t^{\alpha_1}\partial_x^{\alpha_2}\mathbf{u}, \partial_t^{\alpha_1}\partial_x^{\alpha_2}\mathbf{v}\in C([0,T^\varepsilon];H^{s-\alpha_1-\alpha_2}(\mathbb{R}))$.
\end{proof}
The uniform lower bound on $T^\ast$ rests on the structural stability of the relaxation \cite{Yong1993,ChenLevermoreLiu1994}: the sub-characteristic condition \cref{eq:subchar} ensures that \cref{eq:jinxin-sym} admits a strictly convex entropy and a symmetrizer compatible with the relaxation operator (the matrix $\bm{A}_0$ of \cref{thm:total-uniform} is the linearized instance), under which the stiff source is dissipative.  The associated entropy dissipation yields $H^s$ energy estimates whose constants are independent of $\varepsilon$ for $s\ge 1+d/2$; these in turn close the local existence argument on a time interval that does not shrink as $\varepsilon\to0$.  Without \cref{eq:subchar}, $T^\varepsilon$ may collapse, which is why \cref{thm:local-existence} alone gives only an $\varepsilon$-dependent time.

\subsection{Generalization error: approximation theory}\label{subsec:gen}
For a network $(\mathbf{u}_{\theta},\mathbf{v}_{\theta})$, define the interior residuals
\begin{equation}\label{eq:residuals-sym}
\begin{cases}
  \Res_{i,\mathbf{u}}[\theta] = \partial_t \mathbf{u}_{\theta} + \partial_x \mathbf{v}_{\theta}, \\
  \Res_{i,\mathbf{v}}[\theta] = \partial_t \mathbf{v}_{\theta} + \mathbf{A}^2 \partial_x \mathbf{u}_{\theta} - \frac{1}{\varepsilon}\bigl(\mathbf{f}(\mathbf{u}_{\theta}) - \mathbf{v}_{\theta}\bigr), 
\end{cases}
\end{equation}
the initial residuals
\begin{equation}
\begin{cases}
  \Res_{t,\mathbf{u}}[\theta] = \mathbf{u}_{\theta}(0,\cdot) - \mathbf{u}_0, \\
  \Res_{t,\mathbf{v}}[\theta] = \mathbf{v}_{\theta}(0,\cdot) - \mathbf{v}_0, 
\end{cases}
\end{equation}
and the generalization error
\begin{equation}\label{eq:gen-error}
  \mathcal{E}_G[\theta]^2 = \mathcal{E}_G^i[\theta]^2 + \mathcal{E}_G^t[\theta]^2 + \lambda\,\mathcal{E}_G^p[\theta]^2,
\end{equation}
where
\begin{equation*}
\begin{aligned}
  \mathcal{E}_G^i[\theta]^2 &= \int_0^T\!\bigl(\|\Res_{i,\mathbf{u}}[\theta](t,\cdot)\|_{L^2(\mathbb{R})}^2
    + \|\Res_{i,\mathbf{v}}[\theta](t,\cdot)\|_{L^2(\mathbb{R})}^2\bigr)\,\mathrm{d}t, \\
  \mathcal{E}_G^t[\theta]^2 &= \|\Res_{t,\mathbf{u}}[\theta]\|_{L^2(\mathbb{R})}^2
    + \|\Res_{t,\mathbf{v}}[\theta]\|_{L^2(\mathbb{R})}^2, \\
  \mathcal{E}_G^p[\theta]^2 &= \int_0^T\!\bigl(\|\mathbf{u}_{\theta}(t,\cdot)\|_{H^1(\mathbb{R})}^2
    + \|\mathbf{v}_{\theta}(t,\cdot)\|_{H^1(\mathbb{R})}^2\bigr)\,\mathrm{d}t,
\end{aligned}
\end{equation*}
the training error
\begin{equation}\label{eq:train-error}
  \mathcal{E}_T[\theta;\mathcal{S}]^2 = \mathcal{E}_T^i[\theta;\mathcal{S}_i]^2
    + \mathcal{E}_T^t[\theta;\mathcal{S}_t]^2 + \lambda\,\mathcal{E}_T^p[\theta;\mathcal{S}_i]^2,
\end{equation}
where
\begin{equation*}
\begin{aligned}
  \mathcal{E}_T^i[\theta; \mathcal{S}_i]^2 &:= \sum_{(t_n,x_n)\in\mathcal{S}_i} w_n^i \left[ \bigl(\Res_{i,\mathbf{u}}[\theta](t_n,x_n)\bigr)^2 + \bigl(\Res_{i,\mathbf{v}}[\theta](t_n,x_n)\bigr)^2 \right], \\
  \mathcal{E}_T^t[\theta; \mathcal{S}_t]^2 &:= \sum_{x_n\in\mathcal{S}_t} w_n^t \left[ \bigl(\Res_{t,\mathbf{u}}[\theta](x_n)\bigr)^2 + \bigl(\Res_{t,\mathbf{v}}[\theta](x_n)\bigr)^2 \right], \\
  \mathcal{E}_T^p[\theta; \mathcal{S}_i]^2 &:= \sum_{(t_n,x_n)\in\mathcal{S}_i} \sum_{|\alpha|\le 1} w_n^i \left[ \bigl(D^\alpha \mathbf{u}_{\theta}(t_n,x_n)\bigr)^2 + \bigl(D^\alpha \mathbf{v}_{\theta}(t_n,x_n)\bigr)^2 \right],
\end{aligned}
\end{equation*}
and the total error
\begin{equation}\label{eq:total-error}
  \mathcal{E}[\theta]^2 = \int_0^T\!\bigl(\|\mathbf{u}(t,\cdot) - \mathbf{u}_{\theta}(t,\cdot)\|_{L^2(\mathbb{R})}^2
    + \|\mathbf{v}(t,\cdot) - \mathbf{v}_{\theta}(t,\cdot)\|_{L^2(\mathbb{R})}^2\bigr)\,\mathrm{d}t.
\end{equation}

Here, $\lambda>0$ weights the $H^1$ regularization term, and $\mathcal{E}_T^i, \mathcal{E}_T^t, \mathcal{E}_T^p$ are the midpoint-quadrature counterparts (\cref{subsec:quadrature}). The spatiotemporal quadrature points form the dataset $\mathcal{S} = (\mathcal{S}_i, \mathcal{S}_t)$, where $\mathcal{S}_i \subseteq [0,T] \times \mathbb{R}$ and $\mathcal{S}_t \subseteq \mathbb{R}$, with $(w_n^i, w_n^t)$ being the corresponding quadrature weights. We use a Sobolev approximation result for $\tanh$ networks.

\begin{lemma}[$\tanh$ approximation in Sobolev norm \cite{DeRyckNS2024}]
\label{lem:tanh}
Let $d,n\ge2$, $m\ge3$, and $\Omega=\prod_{i=1}^d[a_i,b_i]$.  For $f\in H^m(\Omega)$ and any integer $N>5$ there is a two-hidden-layer $\tanh$ network $\widehat f_N$ such that, for all $k\in\{0,\dots,m-1\}$,
\begin{equation}\label{eq:tanh-rate}
  \norm{f-\widehat f_N}{H^k(\Omega)}
  \le C_{k,m,d,f,\Omega}\,(1+\ln^k N)\,N^{-m+k},
\end{equation}
with layer widths polynomial in $N$ and weights of magnitude $\mathcal{O}(N\ln N+N^\gamma)$.
\end{lemma}

\begin{theorem}[Generalization-error bound]\label{thm:gen-error}
Assume the exact solution of \cref{eq:jinxin-sym} is supported in $[-R,R]$ for all $t\in[0,T]$, set $\Omega=[0,T]\times[-R,R]$, let $\mathbf{f}$ be Lipschitz on its range with constant $L$, and $\lambda=\mathcal{O}(\delta^2)$. Then there is a two-hidden-layer $\tanh$ network $(\mathbf{u}_{\theta},\mathbf{v}_{\theta})$ with
\begin{equation}\label{eq:gen-rate}
  \mathcal{E}_G[\theta] \le
  \begin{cases}
    C\delta, & s=1,\\
    C\,(1+\ln N)\,N^{-s+1}+C\delta, & s\ge2,\ s\in\mathbb{N},
  \end{cases}
\end{equation}
where $\delta>0$ is arbitrary, $N$ controls the width, and $C=C(\mathbf{A},L,\varepsilon,\Omega,\mathbf{u},\mathbf{v})$ is independent of the network size.
\end{theorem}

\begin{proof}
By \cref{thm:spacetime}, $(\mathbf{u},\mathbf{v})\in(H^s([0,T]\times\mathbb{R}))^n\times(H^s([0,T]\times\mathbb{R}))^n$.

For $s=1$, two-layer $\tanh$ networks are dense in $C_c^\infty(\Omega)$, which is dense in $H^1(\Omega)$ \cite{Pinkus1999}. So for any $\delta>0$, there exists $(\mathbf{u}_{\theta},\mathbf{v}_{\theta})$ with $\|\mathbf{u}-\mathbf{u}_{\theta}\|_{H^1}^2\le\delta$ and $\|\mathbf{v}-\mathbf{v}_{\theta}\|_{H^1}^2\le\delta$.

For $s\ge2$, \cref{lem:tanh} with $m=s$, $k=1$, $d=2$ gives $\|\mathbf{u}-\mathbf{u}_{\theta}\|_{H^1},\|\mathbf{v}-\mathbf{v}_{\theta}\|_{H^1} \le C(1+\ln N)N^{-s+1}$.

Set 
\begin{equation*}
  \kappa =
  \begin{cases}
     \delta , & s=1, \\
     C^2(1+\ln N)^2N^{-2s+2} , & s\ge2.
  \end{cases}
\end{equation*}
Thus, $\|\mathbf{u}-\mathbf{u}_{\theta}\|_{H^1}^2,\|\mathbf{v}-\mathbf{v}_{\theta}\|_{H^1}^2\le\kappa$. 

The trace theorem gives $\mathcal{E}_G^t[\theta]^2 \le D\kappa$. For the interior term, expanding $\mathcal{E}_G^i[\theta]^2$ and using $\|\mathbf{f}(\mathbf{u}_{\theta})-\mathbf{f}(\mathbf{u})\|_{L^2}\le L\|\mathbf{u}_{\theta}-\mathbf{u}\|_{L^2}$,
\begin{equation*}
\begin{aligned}
  \mathcal{E}_G^i[\theta]^2
  &\le \Bigl(2+4\|\mathbf{A}\|^4+\tfrac{4L^2}{\varepsilon^2}\Bigr)
      \|\mathbf{u}_{\theta}-\mathbf{u}\|_{H^1}^2
    +\Bigl(6+\tfrac{4}{\varepsilon^2}\Bigr)\|\mathbf{v}_{\theta}-\mathbf{v}\|_{H^1}^2 \\
  &\le \Bigl(8+4\|\mathbf{A}\|^4+\tfrac{4L^2+4}{\varepsilon^2}\Bigr)\kappa .
\end{aligned}
\end{equation*}
Hence $\mathcal{E}_G[\theta]^2\le\bigl(8+4\|\mathbf{A}\|^4+(4L^2+4)/\varepsilon^2 + D \bigr)\kappa + \lambda \mathcal{E}_G^p[\theta]^2$. With the boundedness of the neural network, taking square roots gives \cref{eq:gen-rate}.
\end{proof}

\begin{remark}\label{rem:compact-support}
When neural networks fit functions, polynomial convergence rates in Sobolev norm typically require stronger regularity of the target function. The approximation rate is governed by the regularity of the exact solution: $H^1$ data yield only a qualitative (density) result, whereas $H^s$ data with $s\ge2$ give the algebraic rate $\mathcal{O}(N^{-s+1}\ln N)$. Compact support is natural: data are prescribed in $H^s$; PINNs require truncation of non-compact data; and finite-speed propagation keeps compactly supported data compactly supported on a finite time horizon. We stress that \cref{thm:gen-error} is an \emph{existence} statement: it asserts that some two-hidden-layer $\tanh$ network attains the stated generalization error, and places no constraint on the architecture used in practice. The networks of \cref{sec:numerics} are deeper (\cref{tab:perproblem}) and at least as expressive in practice, although \cref{thm:gen-error} is proved only for the specific two-hidden-layer construction of \cref{lem:tanh}; the gap between attainable and attained error is a question of optimization rather than approximation.
\end{remark}

\begin{remark}\label{rem:eps-dependence}
Constant $C=C(\mathbf{A},L,\varepsilon,\Omega,\mathbf{u},\mathbf{v})$ in \cref{thm:gen-error} depends on $\varepsilon$ through the stiff term: the interior estimate carries a factor $8+4\|\mathbf{A}\|^4+(4L^2+4)/\varepsilon^2$, so for a \emph{fixed} network size the guaranteed generalization error grows like $\varepsilon^{-1}$ as $\varepsilon\to0$. This reflects a genuine fact (approximating the relaxation system at small $\varepsilon$ is harder, since its solution develops $\mathcal{O}(\varepsilon)$ internal layers) and is precisely the difficulty that the continuation strategy is designed to manage: rather than training at a small $\varepsilon$ from scratch, JXRGCM decreases $\varepsilon$ gradually and warm-starts each stage (\cref{alg:jxrgcm}), so the network need only track an incrementally sharper profile. Two points keep this from degrading the final result. First, the rate \cref{eq:gen-rate} shows the error can still be driven to zero at any fixed $\varepsilon$ by increasing the width $N$. Second, and more importantly, the constant $\widetilde C_T$ that propagates the generalization error to the total error in \cref{thm:total-uniform} is \emph{$\varepsilon$-uniform}, so the achieved total error does not blow up along the schedule. In other words, the $\varepsilon$-dependence here bounds the training \emph{effort} required at small $\varepsilon$, not the stability of the error once a target tolerance is met.
\end{remark}

\subsection{From generalization error to total error}\label{subsec:total}
Two important estimates will be obtained in this subsection. The first is simpler but its constant degrades as $\varepsilon\to0$; the second is uniform in $\varepsilon$ under the sub-characteristic condition.

\begin{theorem}\label{thm:total-nonuniform}
Assume the exact solution has compact support and $\mathbf{f}$ is Lipschitz with constant $L$. There is $\widetilde C_{T,\varepsilon}>0$ depending on $T,\mathbf{A},\varepsilon,L$ with $\mathcal{E}[\theta]\le \widetilde C_{T,\varepsilon} \,\mathcal{E}_G[\theta]$.
\end{theorem}
\begin{proof}
Let $\mathbf{e}_{\mathbf{u}}=\mathbf{u}-\mathbf{u}_{\theta}$, $\mathbf{e}_{\mathbf{v}}=\mathbf{v}-\mathbf{v}_{\theta}$, with error system:
\begin{align*}
\begin{cases}
    \partial_t\mathbf{e}_{\mathbf{u}}+\partial_x\mathbf{e}_{\mathbf{v}}=-\Res_{i,\mathbf{u}}[\theta] \\
    \partial_t\mathbf{e}_{\mathbf{v}}+\mathbf{A}^2\partial_x\mathbf{e}_{\mathbf{u}}=\frac{1}{\varepsilon}(\mathbf{f}(\mathbf{u}_{\theta})-\mathbf{f}(\mathbf{u})-\mathbf{e}_{\mathbf{v}})-\Res_{i,\mathbf{v}}[\theta].
\end{cases}
\end{align*}
With $E(t)=\|\mathbf{A}\mathbf{e}_{\mathbf{u}}\|_{L^2(\mathbb{R})}^2+\|\mathbf{e}_{\mathbf{v}}\|_{L^2(\mathbb{R})}^2$, symmetry of $\mathbf{A}$ and integration by parts cancel the spatial cross terms, and
\[
  \frac{1}{2}\frac{\mathrm{d}E}{\mathrm{d}t}
  = -\langle\Res_{i,\mathbf{u}},\mathbf{A}^2\mathbf{e}_{\mathbf{u}}\rangle
    +\frac{1}{\varepsilon}\langle\mathbf{f}(\mathbf{u}_{\theta})-\mathbf{f}(\mathbf{u}),\mathbf{e}_{\mathbf{v}}\rangle
    -\frac{1}{\varepsilon}\|\mathbf{e}_{\mathbf{v}}\|_{L^2}^2-\langle\Res_{i,\mathbf{v}},\mathbf{e}_{\mathbf{v}}\rangle .
\]
Bounding the Lipschitz term by $\frac{LC_{A^{-1}}}{\varepsilon} \|\mathbf{A}\mathbf{e}_{\mathbf{u}}\|_{L^2}\|\mathbf{e}_{\mathbf{v}}\|_{L^2}$ and the residual terms using the Cauchy--Schwarz and Young inequalities, then absorbing quadratics, yields, via Gr\"onwall's inequality~\cite{GripenbergLondenStaffans1990}, $\frac{\mathrm{d}E}{\mathrm{d}t}\le CE+D(\|\Res_{i,\mathbf{u}}\|_{L^2}^2+\|\Res_{i,\mathbf{v}}\|_{L^2}^2)$ with $C,D$ depending on $\mathbf{A},\varepsilon,L$. Gr\"onwall, integration over $[0,T]$, and the equivalence $\|\mathbf{e}_{\mathbf{u}}\|_{L^2}^2+\|\mathbf{e}_{\mathbf{v}}\|_{L^2}^2\le\max\{C_{A^{-1}}^2,1\}E$ give $\mathcal{E}[\theta]\le\widetilde C_{T,\varepsilon}\,\mathcal{E}_G[\theta]$. The constant $\widetilde C_{T,\varepsilon}$ contains $L/\varepsilon$ and degrades as $\varepsilon\to0$.
\end{proof}

\begin{theorem}[$\varepsilon$-uniform total-error estimate]\label{thm:total-uniform}
Assume the sub-characteristic condition \cref{eq:subchar}. Under either: (i) initial data of $H^s$ regularity with $s>3/2$, or (ii) initial data of $H^1$ regularity with uniform-in-$\varepsilon$ existence time and a uniform $W^{1,\infty}$ bound on the solution
\begin{equation}\tag{H}\label{eq:Hunif}
  \|\mathbf{u}^\varepsilon\|_{W^{1,\infty}([0,T]\times\mathbb{R})} \le M,
\end{equation}
with $M$ independent of $\varepsilon$. Then there is $\widetilde C_T>0$ depending only on $T,\mathbf{A},L,\lambda$ (\emph{independent of $\varepsilon$}) with $\mathcal{E}[\theta]\le \widetilde C_T\,\mathcal{E}_G[\theta]$.
\end{theorem}

\begin{proof}
For $H^s(\mathbb{R})\hookrightarrow W^{1,\infty}(\mathbb{R})$ in one dimension and the $\varepsilon$-uniform bound supplied by \cref{thm:uniform-existence}, we only have to prove the theorem under condition (ii).

For simplicity, we present the representative scalar ($2\times2$) case. Writing $\bm{E}=(e_u,e_v)^\top$ and using $f(u^\varepsilon)-f(u^\varepsilon-e_u)=\bar f'(U^\varepsilon)e_u$ with $\bar f'(U^\varepsilon)=\int_0^1 f'(u^\varepsilon-\xi e_u)\,\mathrm{d}\xi$, the error system becomes
\begin{equation}\label{eq:error-matrix}
  \bm{E}_t+\bm{A}_1\bm{E}_x=\tfrac1\varepsilon\bm{Q}(U^\varepsilon)\bm{E}+\bm{R},
\end{equation}
where
\begin{equation*}
  \bm{A}_1=\begin{bmatrix}0&1\\ a^2&0\end{bmatrix}, \quad
  \bm{Q}=\begin{bmatrix}0&0\\ -\bar f'& -1\end{bmatrix}, \quad
  \bm{R}=\begin{bmatrix}-\Res_{i,u}\\ -\Res_{i,v}\end{bmatrix}.
\end{equation*}

Introduce the symmetrizer $\bm{A}_0(U^\varepsilon)=\begin{bmatrix}a^2 & \bar f'\\ \bar f' & 1\end{bmatrix}$. Under \cref{eq:subchar}, $|\bar f'|<a$, so $\bm{A}_0$ is symmetric positive definite with $\det\bm{A}_0=a^2-(\bar f')^2>0$, hence uniformly equivalent to the identity: $C_0^{-1}|\bm{x}|^2\le\bm{x}^\top\bm{A}_0\bm{x}\le C_0|\bm{x}|^2$ with $C_0$ independent of $\varepsilon$. Define the energy $E_{\mathrm{e}}(t)=\tfrac12\int_\mathbb{R}\bm{E}^\top\bm{A}_0\bm{E}\,\mathrm{d}x$, which is equivalent to $\|\bm{E}\|_{L^2(\mathbb{R})}^2$. The stiff term is dissipative: a direct computation gives
\[
  \bm{A}_0\bm{Q}+\bm{Q}^\top\bm{A}_0
  =\begin{bmatrix}-2(\bar f')^2 & -2\bar f'\\ -2\bar f' & -2\end{bmatrix},
  \qquad
  \bm{E}^\top\bigl(\bm{A}_0\bm{Q}+\bm{Q}^\top\bm{A}_0\bigr)\bm{E}
  = -2\bigl(\bar f'\,e_u+e_v\bigr)^2 \le 0,
\]
so this matrix is negative semidefinite (its determinant vanishes; the controlled direction is $\bar f' e_u+e_v$). Since $\bm{A}_0$ is symmetric, $\bm{A}_0\bm{Q}$ has symmetric part $\tfrac12(\bm{A}_0\bm{Q}+\bm{Q}^\top\bm{A}_0)$, so $\bm{E}^\top\bm{A}_0\bm{Q}\bm{E} =\tfrac12\bm{E}^\top(\bm{A}_0\bm{Q}+\bm{Q}^\top\bm{A}_0)\bm{E}$ and
\[
  \tfrac1\varepsilon\!\int_\mathbb{R}\bm{E}^\top\bm{A}_0\bm{Q}\bm{E}\,\mathrm{d}x
  = \tfrac1{2\varepsilon}\!\int_\mathbb{R}
    \bm{E}^\top\bigl(\bm{A}_0\bm{Q}+\bm{Q}^\top\bm{A}_0\bigr)\bm{E}\,\mathrm{d}x
  = -\tfrac1\varepsilon\!\int_\mathbb{R}\bigl(\bar f'\,e_u+e_v\bigr)^2\,\mathrm{d}x \le 0 .
\]
The convection coefficient is symmetric,
\[
  \bm{A}_0\bm{A}_1
  =\begin{bmatrix}a^2 & \bar f'\\ \bar f' & 1\end{bmatrix}
   \begin{bmatrix}0 & 1\\ a^2 & 0\end{bmatrix}
  =\begin{bmatrix}a^2\bar f' & a^2\\ a^2 & \bar f'\end{bmatrix}
  =(\bm{A}_0\bm{A}_1)^\top,
\]
so integration by parts moves the spatial derivative onto the matrix field,
\[
  -\int_\mathbb{R}\bm{E}^\top\bm{A}_0\bm{A}_1\bm{E}_x\,\mathrm{d}x
  =\tfrac12\int_\mathbb{R}\bm{E}^\top\partial_x(\bm{A}_0\bm{A}_1)\bm{E}\,\mathrm{d}x .
\]
Hence
\[
  \tfrac{\mathrm{d}}{\mathrm{d}t}E_{\mathrm{e}}(t)
  \le\tfrac12\!\int_\mathbb{R}\bm{E}^\top\bm{M}(U^\varepsilon)\bm{E}\,\mathrm{d}x
    +\int_\mathbb{R}\bm{E}^\top\bm{A}_0\bm{R}\,\mathrm{d}x,
  \quad
  \bm{M}=\partial_t\bm{A}_0+\partial_x(\bm{A}_0\bm{A}_1).
\]
The entries of $\bm{M}$ involve only $\partial_{t,x}\bar f'$, controlled in $L^\infty$ by a constant $C_{\bm{M}}$ \emph{independent of $\varepsilon$} via hypothesis \textup{(H)} and $f\in C^\infty$. Therefore $\tfrac{\mathrm{d}}{\mathrm{d}t}E_{\mathrm{e}}\le C_3 E_{\mathrm{e}}+C_2\|\bm{R}\|_{L^2(\mathbb{R})}^2$ with $C_2,C_3$ independent of $\varepsilon$. Using Gr\"onwall inequality and absorbing $T,M,a,L$-dependent constants into $\widetilde C_T$ yields $\mathcal{E}[\theta]\le \widetilde C_T\,\mathcal{E}_G[\theta]$.
\end{proof}

\subsection{Relaxation limit and convergence of the method}\label{subsec:limit}
\begin{lemma}[Jin--Xin relaxation limit \cite{Natalini1996}]\label{lem:relax-limit}
Let the flux function $f \colon \mathbb{R} \to \mathbb{R}$ be of class $C^1$ and satisfy $f(0) = f'(0) = 0$. Assume that the perturbed initial data $(u_0^\varepsilon, v_0^\varepsilon) \in (L^1(\R) \cap L^\infty(\R) \cap BV(\R))^2$ takes the form
\begin{equation}\label{eq:init-data}
    u_0^\varepsilon(x) := u_0(x), \quad v_0^\varepsilon(x) = f(u_0(x)) + K(x)\omega(\varepsilon),
\end{equation}
where $K \in L^1(\R) \cap L^\infty(\R) \cap BV(\R)$, and $\omega: [0,\infty) \to [0,\infty)$ is a continuous function satisfying $\omega(0) = 0$.

Furthermore, assume there exist positive constants $\rho_0 > 0$ and $M > 0$, independent of $\varepsilon$, such that the following uniform bounds hold:
\begin{equation}\label{eq:uniform-bounds}
    \rho_0 = \max \left( \sup_{\varepsilon>0} \|v_0^\varepsilon\|_{L^\infty}, \sup_{\varepsilon>0} \|u_0^\varepsilon\|_{L^\infty} \right), \quad 
    \|(u_0^\varepsilon, v_0^\varepsilon)\|_{BV} \le M,
\end{equation}
and for the Lipschitz constant of $f$ and the kernel $K$:
\begin{equation}\label{eq:lip-k-bounds}
    \text{Lip}(f) := \sup_{x \neq y} \left| \frac{f(x) - f(y)}{x - y} \right| \le M, \quad \|K\|_{L^1} \le M.
\end{equation}

Under these assumptions, the following results hold:
\begin{enumerate}
    \item \textbf{Convergence:} As the relaxation rate $\varepsilon \to 0$, the global solution $(u^\varepsilon, v^\varepsilon)$ converges to the equilibrium state $(u, f(u))$.
    
    \item \textbf{Error estimate for $u^\varepsilon$:} For any fixed time horizon $T > 0$ and for all $t \in [0, T]$, there exists a positive constant $C_T$ depending only on $T$ such that
    \begin{equation}\label{eq:u-error}
        \|u^\varepsilon(\cdot,t) - u(\cdot,t)\|_{L^1(\R)} \le C_T \sqrt{\varepsilon}.
    \end{equation}
    
    \item \textbf{Local $BV$ estimate:} If the initial data $(u_0^\varepsilon, v_0^\varepsilon)$ is locally of bounded variation, then for any interval $(a, b) \subseteq \mathbb{R}$ and any $t \ge 0$, there exists a constant $C = C(\alpha)$ depending only on $\alpha$ and independent of $\varepsilon$ such that
    \begin{equation}\label{eq:local-bv}
        \|(u^\varepsilon(\cdot,t), v^\varepsilon(\cdot,t))\|_{BV((a,b))} \le C \|(u_0^\varepsilon, v_0^\varepsilon)\|_{BV((a-\sqrt{\alpha t}, b+\sqrt{\alpha t}))}.
    \end{equation}
\end{enumerate}
\end{lemma}

\begin{theorem}[Convergence of JXRGCM]\label{thm:overall}
Let $u^0$ be the entropy solution in the smooth regime $t\in[0,T]$ and let $u_{\mathrm{NN}}^\varepsilon$ be the PINN trained on the Jin--Xin system at parameter $\varepsilon$.  Under the hypotheses of \cref{thm:total-uniform,lem:relax-limit}, there are constants $C_1,C_2>0$ with $C_1$ independent of $\varepsilon$ such that
\begin{equation}\label{eq:overall-bound}
  \norm{u_{\mathrm{NN}}^\varepsilon-u^0}{L^2([0,T]\times\R)}
  \le C_1\,\EG[\theta] + C_2\,\varepsilon^{1/4}.
\end{equation}
In particular $u_{\mathrm{NN}}^\varepsilon\to u^0$ in $L^2$ as $\EG[\theta]\to0$ and $\varepsilon\to0$. Under the hypotheses of \cref{thm:total-nonuniform,lem:relax-limit} we have a similar result with a uniform estimate of $\varepsilon$.
\end{theorem}
\begin{proof}
Split by the triangle inequality,
\[
  \norm{u_{\mathrm{NN}}^\varepsilon-u^0}{L^2}
  \le\norm{u_{\mathrm{NN}}^\varepsilon-u^\varepsilon}{L^2}
   +\norm{u^\varepsilon-u^0}{L^2}.
\]
The first term is the PINN approximation error, bounded by $C_1\EG[\theta]$ via \cref{thm:total-uniform} (or \cref{thm:total-nonuniform}).  For the second, the solutions are uniformly $L^\infty$-bounded (from the $BV$ estimate), so for fixed $t$,
\[
  \norm{u^\varepsilon-u^0}{L^2(\R)}^2
  \le\norm{u^\varepsilon-u^0}{L^\infty}\,\norm{u^\varepsilon-u^0}{L^1}
  \le 2M_\infty C_T\sqrt{\varepsilon}.
\]
Integrating over $[0,T]$ and taking the square root yields $\norm{u^\varepsilon-u^0}{L^2([0,T]\times\R)}\le C_2\,\varepsilon^{1/4}$ with $C_2=\sqrt{2M_\infty C_T T}$, which gives \cref{eq:overall-bound}.
\end{proof}

\begin{remark}\label{rem:scalar-scope}
\Cref{thm:overall} should be read as a \emph{scalar} result.  The relaxation-error step relies on \cref{lem:relax-limit} (Natalini's $L^1$ estimate \cite{Natalini1996}), which is established for scalar conservation laws with uniformly bounded $BV$ solutions; the $L^2$ rate $\mathcal{O}(\varepsilon^{1/4})$ then follows by $L^1$--$L^\infty$ interpolation.  For systems such as the shallow-water or Euler equations a corresponding $BV$/$L^1$ relaxation-limit theory with an explicit rate is not available in general, so \cref{eq:overall-bound} does not extend to them automatically: the PINN-to-relaxation bound (\cref{thm:total-uniform}) still holds system-wise, but the relaxation-to-equilibrium rate does not.  The system experiments of \cref{sec:numerics} are therefore reported as empirical evidence rather than as instances of \cref{thm:overall}.  Both results are also smooth-regime statements: the rate holds while $u^0$ stays in the regime covered by \cref{lem:relax-limit}. The $\varepsilon$-uniformity of $C_1$ (\cref{thm:total-uniform}) is what makes the continuation meaningful: the network error does not blow up as the schedule drives $\varepsilon\to0$.
\end{remark}

\subsection{Training error and quadrature}\label{subsec:quadrature}
The stability analysis is formulated in terms of the generalization error $\EG$, whereas the numerical training procedure is based on residuals evaluated at finitely many collocation points. The purpose of this subsection is to quantify the quadrature error arising from this finite-sample approximation and thereby relate the training and generalization errors.

Computing $\EG[\theta]$ requires numerical quadrature, which is exactly the training error $\ET[\theta;\mathcal{S}]$.  For midpoint quadrature on a domain $\Omega\subset\R^d$ partitioned into $M$ cubes of side $l$ (so $Ml^d=|\Omega|$) and $f\in C^2(\Omega)$, $\lvert\int_\Omega f-Q[f,\Omega]\rvert\le\tfrac{d|\Omega|}{24}l^2 \max_\Omega\norm{\nabla^2 f}{\infty}$ \cite{MishraMolinaro2023,HuLinRaydan2022}. This yields the following gap.

\begin{theorem}\label{thm:quadrature}
Let $\Omega_L=[0,T]\times[-L,L]$ and $\Omega_t=[-L,L]$, with $\mathcal{S}_i$ the midpoints of $M_i$ square cells of $\Omega_L$ and $\mathcal{S}_t$ the midpoints of $M_t$ intervals of $\Omega_t$, and constant weights $w_n^i=2LT/M_i$, $w_n^t=2L/M_t$.  Then for any $\lambda>0$,
\begin{equation}\label{eq:quadrature}
  \bigl\lvert\EG[\theta]-\ET[\theta;\mathcal{S}]\bigr\rvert
  \le\bigl(C_i M_i^{-1}+C_t M_t^{-2}+\lambda C_p M_i^{-1}\bigr)^{1/2},
\end{equation}
with
\begin{align*}
  C_i&=\tfrac{(LT)^2}{3}\norm{\Res_{i,u}^2+\Res_{i,v}^2}{W^{2,\infty}}, \\
  C_t&=\tfrac{L^3}{3}\norm{\Res_{t,u}^2+\Res_{t,v}^2}{W^{2,\infty}}, \\
  C_p&=\tfrac{(LT)^2}{3}\norm{\textstyle\sum_{|\alpha|\le1}
       \bigl((D^\alpha\bu_\theta)^2+(D^\alpha\bv_\theta)^2\bigr)}{W^{2,\infty}}.
\end{align*}
\end{theorem}
\begin{proof}
Expand $|\EG^2-\ET^2|\le|\EG^{i\,2}-\ET^{i\,2}|+|\EG^{t\,2}-\ET^{t\,2}| +\lambda|\EG^{p\,2}-\ET^{p\,2}|$ and apply the midpoint error bound to each term on $\Omega_L$ (with $d=2$, $l_i^2=|\Omega_L|/M_i$) and on $\Omega_t$ (with $d=1$, $l_t=|\Omega_t|/M_t$), which gives the three contributions above; the left side is then bounded by $|\EG-\ET|^2\le|\EG^2-\ET^2|$.
\end{proof}

Thus, for a fixed trained network, the theorem provides an a posteriori connection between the training error and the generalization error entering the stability estimate, with the discrepancy decreasing as the quadrature resolution is refined.

%======================================================================
\section{Numerical experiments}\label{sec:numerics}
We illustrate JXRGCM on three benchmarks for \cref{eq:cons-law}.  Unless noted, errors are root mean square error (RMSE) against the exact solution. $\mathbf{A} = a\mathbf{I}$ represents a special case with a constant scalar speed $a$. In the following numerical experiments, we restrict our attention to this setting throughout.

\FloatBarrier
\subsection{Implementation and continuation schedule}\label{subsec:implementation}
The conserved variables and relaxation fluxes are represented by a single fully connected $\tanh$ network of width $W$ and depth $D$, trained with the Adam optimizer and a per-stage learning-rate schedule (\cref{alg:jxrgcm}). The continuation parameters used in the experiments below are summarized in \cref{tab:schedule}: $\varepsilon$ starts at $\varepsilon_0=10^{-2}$ and is halved ($r=\tfrac12$) at each stage until it falls below $\varepsilon^\ast=10^{-5}$, which takes ten stages and leaves $\varepsilon\approx1.95\times10^{-5}$ as the smallest value used; a final stage at $\varepsilon=0$ follows.  The stage tolerance tightens as $L_{\min}=L\,\varepsilon/\varepsilon_0$, and parameters are warm-started across stages.  All weight parameters are uniformly set to $1$. The relaxation speed $a$ is chosen to satisfy the sub-characteristic condition \cref{eq:subchar} on each problem's range of states. The problem-dependent network and training settings (relaxation speed, width, depth, learning rate, weight parameters and the numbers of interior and initial/boundary sample points) are listed per benchmark in \cref{tab:perproblem}.

\begin{table}[htbp]
  \centering
  \caption{Continuation schedule used in \cref{sec:numerics}.  Starting from
    $\varepsilon_0=10^{-2}$, the relaxation parameter is halved at each stage down
    to $\varepsilon^\ast$, with the stage tolerance $L_{\min}$ tightened in
    proportion to $\varepsilon$ and the network warm-started across stages; a
    final stage at $\varepsilon=0$ sharpens the solution near the limit.}
  \label{tab:schedule}
  \smallskip
  \begin{tabular}{lll}
    \toprule
    Symbol & Meaning & Value \\
    \midrule
    $\varepsilon_0$ & initial relaxation parameter & $10^{-2}$ \\
    $r$ & geometric decay factor & $1/2$ \\
    $\varepsilon^\ast$ & continuation stopping value & $10^{-5}$ \\
    \#stages & number of $\varepsilon$-stages (plus final $\varepsilon=0$) & $10$ \\
    $L_{\min}$ & stage loss tolerance & $L\,\varepsilon/\varepsilon_0$ \\
    optimizer & -- & Adam \\
    $\lambda_{data}$ & the weight parameter for $\mathcal{L}_{\mathrm{data}}$& 1 \\
    $\lambda_{pde}$ & the weight parameter for $\mathcal{L}_{\mathrm{pde}}$ & 1 \\
    \bottomrule
  \end{tabular}
\end{table}

\begin{table}[htbp]
  \centering
  \caption{Per-benchmark network and training settings.  For each problem the
    table lists the relaxation speed $a$ (chosen to satisfy the
    sub-characteristic condition \cref{eq:subchar} on that problem's range of
    states), the width and depth of the shared $\tanh$ network, the base learning
    rate $\eta_0$, and the numbers of interior and initial/boundary sample
    points.}
  \label{tab:perproblem}
  \smallskip
  \small
  \setlength{\tabcolsep}{5pt}
  \begin{tabular}{lccccc}
    \toprule
    Problem & $a$ & width & depth & $\eta_0$
      & $N_{\mathrm{int}}/N_{\mathrm{ic}}$ \\
    \midrule
    Burgers
      & $5$ & $160$ & $3$
      & $0.3$ & $4000/1000$ \\
    Shallow water
      & $5$ & $160$ & $3$
      & $0.5$ & $40000/10000$ \\
    Sod
      & $5$ & $512$ & $3$
      & $0.5$ & $40000/10000$ \\
    \bottomrule
  \end{tabular}
\end{table}

%\FloatBarrier
\subsection{Burgers equation}\label{subsec:burgers}
On $\Omega=[0,1]\times[-1,1]$ we solve $\partial_t u+\partial_x f(u)=0$ with
$f(u)=\tfrac12u^2$ and data
\begin{equation}\label{eq:burgers-ic}
  u_0(x)=
  \begin{cases}
    \phantom{-}1, & -0.5<x<0.5,\\
    -1, & -1<x<-0.5\ \text{or}\ 0.5<x<1 .
  \end{cases}
\end{equation}
For $t\in[0,1)$ (before the rarefaction fan meets the stationary shock at
$x=\tfrac12$) the exact solution is
\begin{equation}\label{eq:burgers-exact}
  u(x,t)=
  \begin{cases}
    -1, & x<-\tfrac12-t,\\[2pt]
    \dfrac{x+\tfrac12}{t}, & -\tfrac12-t<x<-\tfrac12+t,\\[6pt]
    \phantom{-}1, & -\tfrac12+t\le x\le\tfrac12,\\[2pt]
    -1, & x>\tfrac12 .
  \end{cases}
\end{equation}
The relaxation system is:
\begin{equation}
    \begin{cases}
        \partial_t u+\partial_x v=0, \\
\varepsilon(\partial_t v+a^2\partial_x u)=\tfrac12u^2-v. 
    \end{cases}
\end{equation}
\Cref{fig:burgers}
compares JXRGCM with RelaxNN \cite{ZhouMa2024}, which fixes $\varepsilon$ throughout training.  JXRGCM attains $\mathrm{RMSE}\approx3.79\times10^{-2}$, while the fixed-$\varepsilon$ RelaxNN attains an RMSE of approximately $3.00\times10^{-1}$ and fails to resolve the sharp jump in the solution. This performance gap arises because JXRGCM gradually drives the relaxation parameter $\varepsilon$ to zero over the course of computation, whereas RelaxNN maintains $\varepsilon$ fixed at a constant non-zero value throughout.  Subject to the persistent non-zero $\varepsilon$ constraint, the continuous network cannot represent the sharp physical discontinuity, so the fixed-$\varepsilon$ solution remains excessively smeared near the wave fronts and exhibits an RMSE roughly eight times larger, with visible deviations in the shock region. \Cref{fig:burgers loss} shows the evolution of the training loss over the continuation process for the Burgers problem. Each reduction of $\varepsilon$ produces a brief increase in the loss as the target relaxation profile sharpens, followed by a rapid recovery from the warm start and a renewed decrease in the loss.

\begin{figure}[htbp]
  \centering
  \begin{subfigure}[b]{0.49\textwidth}\centering
    \includegraphics[width=\textwidth]{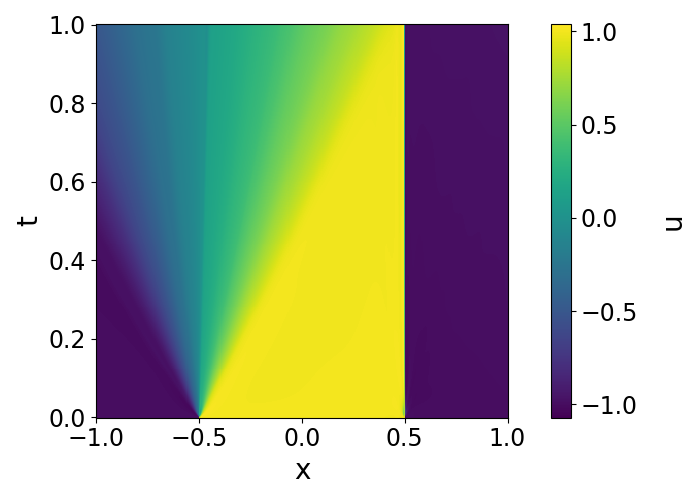}
    \caption{$u$: JXRGCM}\label{fig:burgers-jxrgcm-u}\end{subfigure}
  \begin{subfigure}[b]{0.49\textwidth}\centering
    \includegraphics[width=\textwidth]{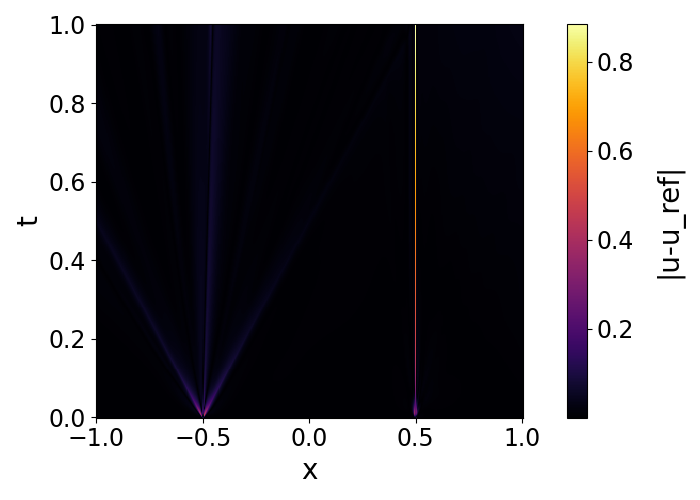}
    \caption{$u$: JXRGCM, absolute error}\label{fig:burgers-jxrgcm-err}\end{subfigure}

  \vspace{0.5em}
  \begin{subfigure}[b]{0.49\textwidth}\centering
    \includegraphics[width=\textwidth]{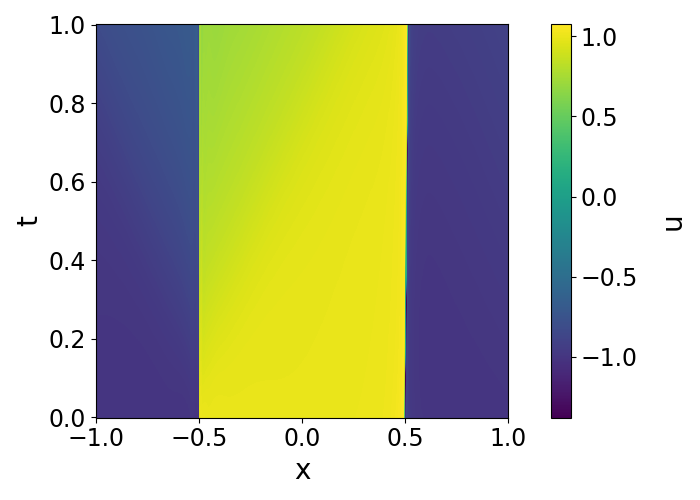}
    \caption{$u$: RelaxNN}\label{fig:burgers-relaxnn-u}\end{subfigure}
  \begin{subfigure}[b]{0.49\textwidth}\centering
    \includegraphics[width=\textwidth]{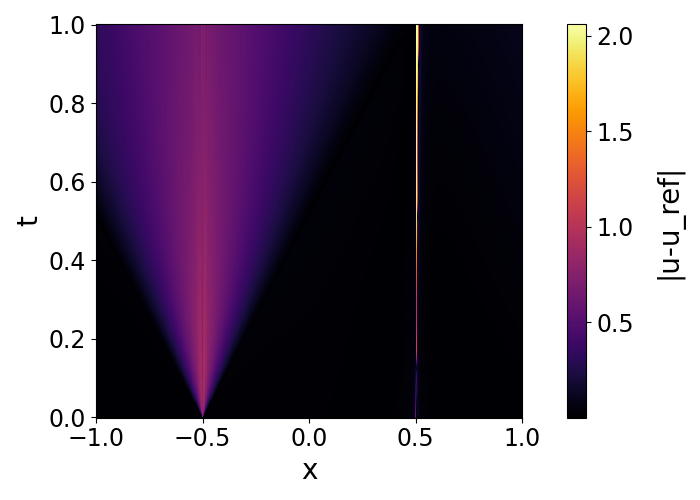}
    \caption{$u$: RelaxNN, absolute error}\label{fig:burgers-relaxnn-err}\end{subfigure}
  \caption{Burgers equation: JXRGCM (top) versus the fixed-$\varepsilon$ RelaxNN
    (bottom).  Left panels show the predicted solution $u(t,x)$ over the
    space--time domain and right panels the absolute error against the exact
    solution \cref{eq:burgers-exact}.  JXRGCM keeps the rarefaction fan and the
    stationary shock sharp, with error concentrated in a thin band along the
    fronts, whereas the fixed-$\varepsilon$ RelaxNN smears the discontinuity and
    spreads error through the interior.}
  \label{fig:burgers}
\end{figure}

\begin{figure}[htbp]
    \centering
    \includegraphics[width=0.9\linewidth]{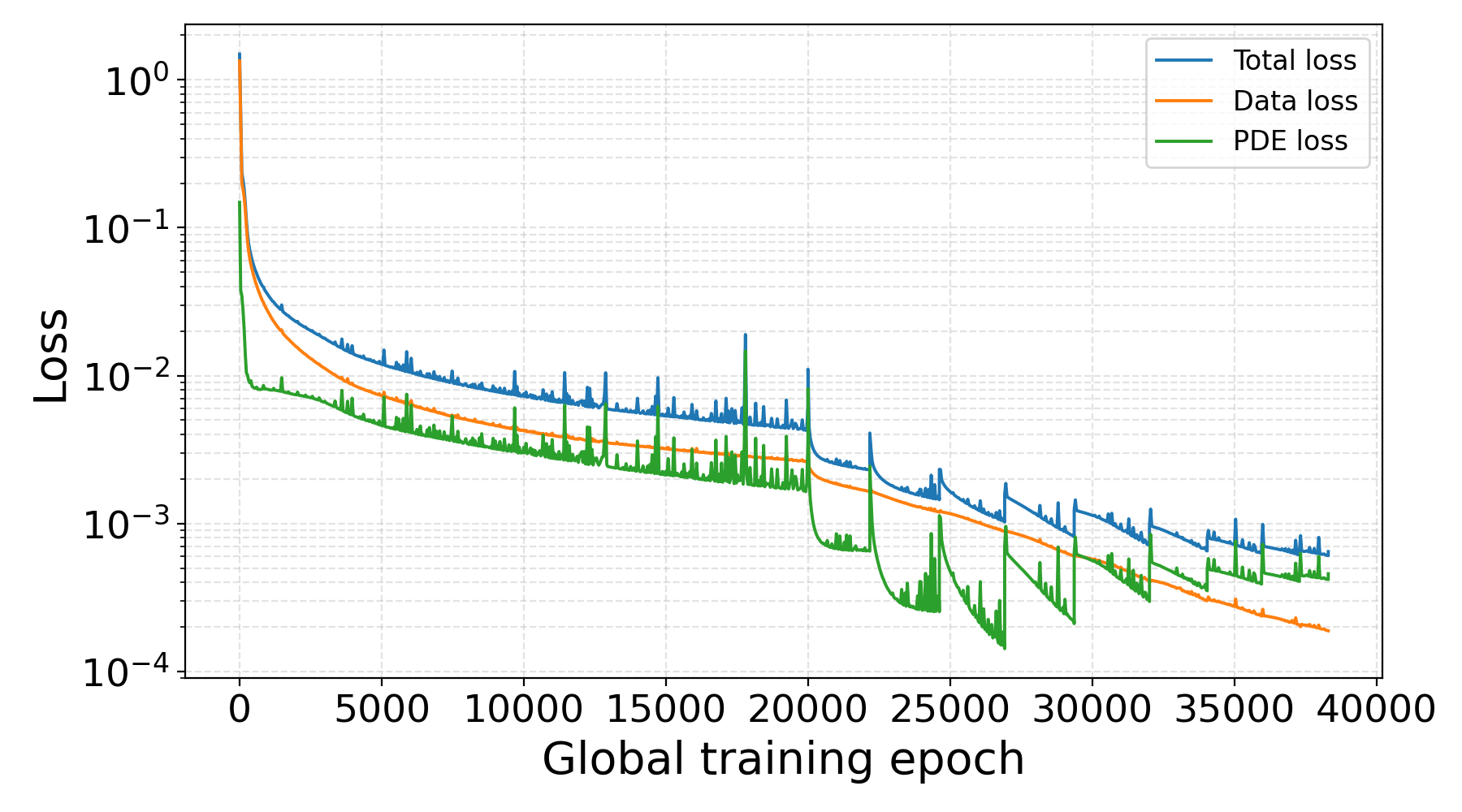}
    \caption{The convergence curves of the loss function for the Burgers problem. Each reduction of $\varepsilon$ produces a short-lived spike in the loss, reflecting the sudden sharpening of the relaxation profile at the start of a new continuation stage. The network then rapidly readjusts from the warm start, and the loss resumes its downward trend over the subsequent iterations.}
    \label{fig:burgers loss}
\end{figure}

\subsection{Shallow-water dam-break problem}\label{subsec:shallow-water}
On $\Omega=[0,1]\times[-1,1]$ we solve $\partial_t h+\partial_x(hu)=0$, $\partial_t(hu)+\partial_x(hu^2+gh^2/2)=0$ with $g=9.81$ and dam-break data $(h,u)=(1,0)$ for $x\le0$, $(0.5,0)$ for $x>0$.  The exact solution is a left rarefaction, a constant middle state $(h^\ast,u^\ast)=(0.727,0.923)$, and a right shock:
\begin{equation}\label{eq:swe-exact-h}
  h(x,t)=
  \begin{cases}
    1, & x\le-3.132\,t,\\[2pt]
    \dfrac{(6.264-x/t)^2}{9g}, & -3.132\,t<x<-1.747\,t,\\[6pt]
    0.727, & -1.747\,t\le x\le2.958\,t,\\[2pt]
    0.5, & x>2.958\,t,
  \end{cases}
\end{equation}
\begin{equation}\label{eq:swe-exact-u}
  u(x,t)=
  \begin{cases}
    0, & x\le-3.132\,t,\\[2pt]
    \frac{2}{3}\bigl(x/t+\sqrt{g\,h_L}\bigr)=\frac{2x}{3t}+2.088,
      & -3.132\,t<x<-1.747\,t,\\[6pt]
    0.923, & -1.747\,t\le x\le2.958\,t,\\[2pt]
    0, & x>2.958\,t,
  \end{cases}
\end{equation}
the constants being the wave speeds of the exact Riemann solution for $g=9.81$ (head $-\sqrt{gh_L}=-3.132$, tail $u^\ast-\sqrt{gh^\ast}=-1.747$, shock speed $2.958$).  With relaxation fluxes $V=(v_1,v_2)$ for the fluxes of $(h,m)$, $m=hu$,
\begin{equation}\label{eq:swe-relax}
  \left\{
  \begin{aligned}
    \partial_t h+\partial_x v_1 &= 0, &
    \varepsilon(\partial_t v_1+a^2\partial_x h) &= m-v_1,\\
    \partial_t m+\partial_x v_2 &= 0, &
    \varepsilon(\partial_t v_2+a^2\partial_x m) &= (hu^2+\tfrac12gh^2)-v_2 .
  \end{aligned}
  \right.
\end{equation}
JXRGCM attains $\mathrm{RMSE}\approx1.17\times10^{-2}$ for $h$ and $4.16\times10^{-2}$ for $u$ (\cref{fig:shallow-water}). \Cref{fig:dam-break loss} shows the evolution of the training loss throughout the continuation process for the shallow-water dam-break problem. Reductions in $\varepsilon$ lead to transient increases in the loss as the network adjusts to progressively sharper wave structures, after which the loss decreases again and the training remains well behaved across the continuation stages.

\begin{figure}[htbp]
  \centering
  \begin{subfigure}[b]{0.49\textwidth}\centering
    \includegraphics[width=\textwidth]{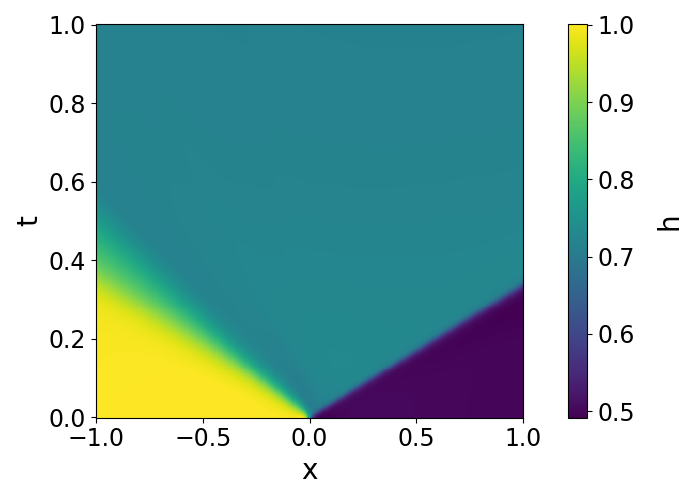}
    \caption{predicted $h$}\label{fig:sw-h}\end{subfigure}
  \begin{subfigure}[b]{0.49\textwidth}\centering
    \includegraphics[width=\textwidth]{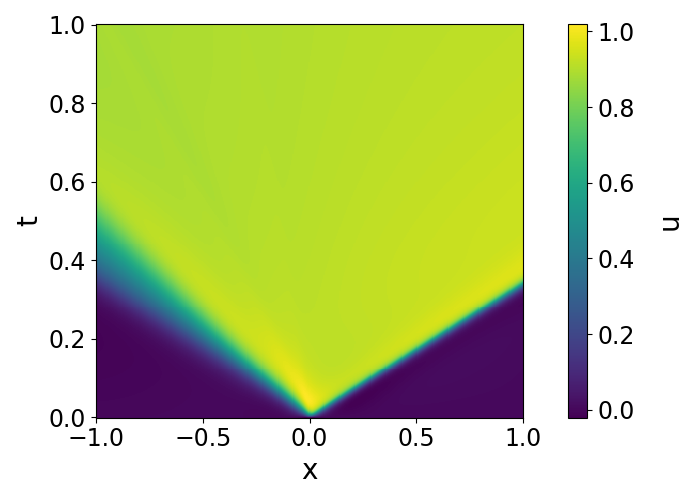}
    \caption{predicted $u$}\label{fig:sw-u}\end{subfigure}

  \vspace{0.5em}
  \begin{subfigure}[b]{0.49\textwidth}\centering
    \includegraphics[width=\textwidth]{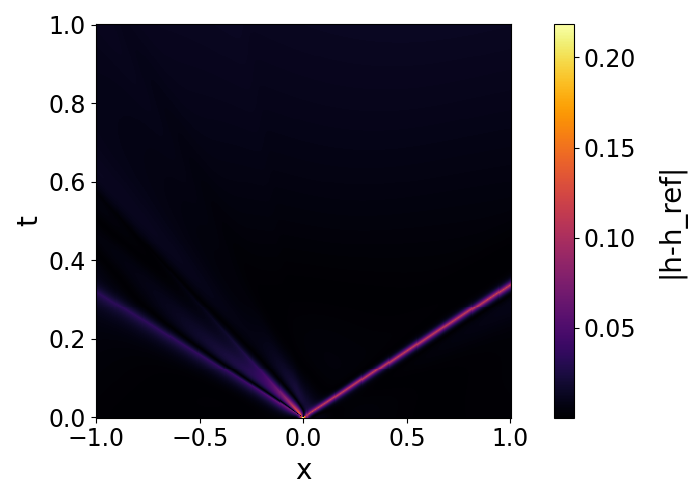}
    \caption{$h$, absolute error}\label{fig:sw-h-err}\end{subfigure}
  \begin{subfigure}[b]{0.49\textwidth}\centering
    \includegraphics[width=\textwidth]{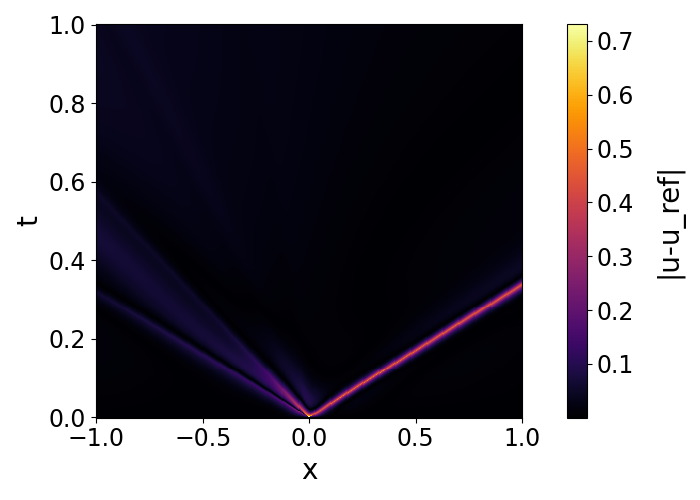}
    \caption{$u$, absolute error}\label{fig:sw-u-err}\end{subfigure}
  \caption{Shallow-water dam break: JXRGCM predictions and absolute errors.  Top
    panels show the predicted height $h$ and velocity $u$; bottom panels show the
    absolute error against the exact Riemann solution
    \cref{eq:swe-exact-h,eq:swe-exact-u}.  The left rarefaction, constant middle
    state, and right shock are all captured, with the largest errors localized
    along the shock and the edges of the rarefaction.}
  \label{fig:shallow-water}
\end{figure}

\begin{figure}[htbp]
    \centering
    \includegraphics[width=0.9\linewidth]{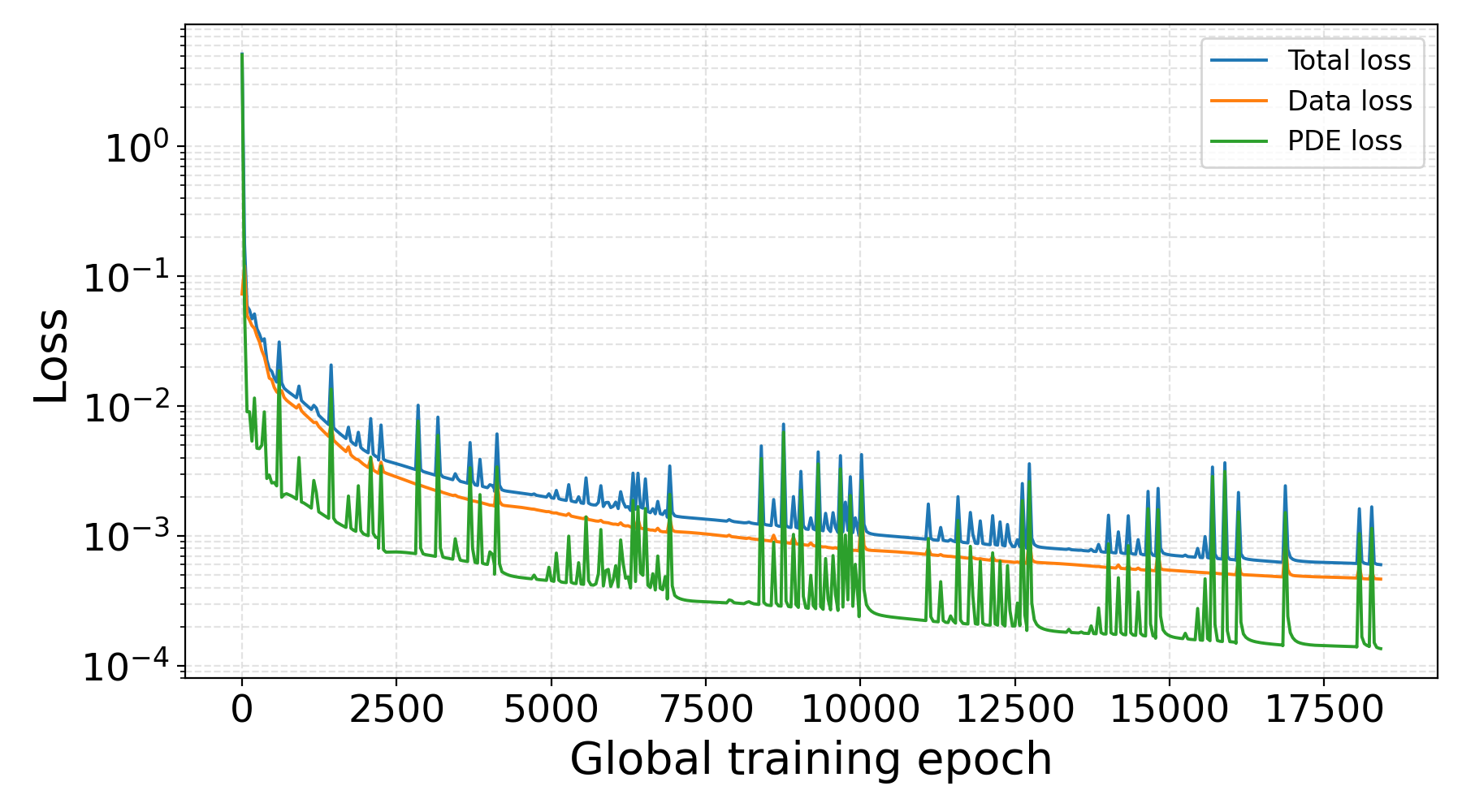}
    \caption{The convergence curves of the loss function for the shallow-water dam-break problem. The peaks correspond to the stages at which $\varepsilon$ is reduced, causing a temporary increase in the loss as the network adapts to a sharper relaxation profile. After each transition, the loss decreases again and the training remains stable throughout the continuation process.}
    \label{fig:dam-break loss}
\end{figure}
\subsection{Sod shock tube problem}\label{subsec:sod}
On $\Omega=[0,0.25]\times[-0.5,0.5]$ we solve the Euler equations
$\partial_t\rho+\partial_x(\rho u)=0$,
$\partial_t(\rho u)+\partial_x(\rho u^2+p)=0$,
$\partial_t E+\partial_x(u(E+p))=0$ with $E=\tfrac{p}{\gamma-1}+\tfrac12\rho u^2$,
$\gamma=1.4$, and Riemann data $(\rho_L,u_L,p_L)=(3,0,3)$,
$(\rho_R,u_R,p_R)=(1,0,1)$.  The exact solution is a left rarefaction, a contact
discontinuity, and a right shock, with $p_\ast=1.693$, $u_\ast=0.464$,
$a_L=\sqrt{\gamma p_L/\rho_L}=1.183$, $S_{HL}=-1.183$, $S_{TL}=-0.626$,
$S_R=1.494$, $\rho_{L,\ast}=1.993$, $\rho_{R,\ast}=1.451$, from the exact
Riemann solver \cite{Toro2009}.  With $\mathbf{u} = (\rho, m, E)^\top$, $m = \rho u$, and the relaxation flux vector $\mathbf{v} = (v_1, v_2, v_3)^\top$, the Jin--Xin relaxation system for the Euler equations is given by
\begin{equation}\label{eq:euler-relax}
  \partial_t \mathbf{u} + \partial_x \mathbf{v} = 0, \quad
  \varepsilon (\partial_t \mathbf{v} + a^2 \partial_x \mathbf{u}) = \mathbf{f}(\mathbf{u}) - \mathbf{v},
\end{equation}
where $\mathbf{f}(\mathbf{u}) = (m, \, \rho u^2 + p, \, u(E + p))^\top$ is the physical flux function.
\Cref{fig:sod loss} demonstrates the convergence curves of the loss function for the Sod shock tube problem.
\Cref{tab:sod} reports the root mean square error (RMSE) of the artificial-viscosity PINN \cite{VonNeumannRichtmyer1950,GhoreishiNaderan2026}, the gradient-annihilating PINN (GA-PINN) \cite{ferrer2024gradient}, the coupled-integral PINN (CI-PINN) \cite{WangYang2024}, and JXRGCM. JXRGCM yields the smallest error across all evaluation metrics, outperforming other physics-informed neural network architectures. To ensure a fair and rigorous comparison among all neural network-based methods (Artificial Viscosity, GA-PINN, CI-PINN, and JXRGCM), all models were evaluated under identical experimental conditions. Specifically, they shared the exact same network architecture (same depth and width), the same number of spatio-temporal collocation points, and were trained using the same optimization protocol (optimizer and learning rate schedule).

\begin{figure}[htbp]
  \centering
  \begin{subfigure}[b]{0.32\textwidth}\centering
    \includegraphics[width=\textwidth]{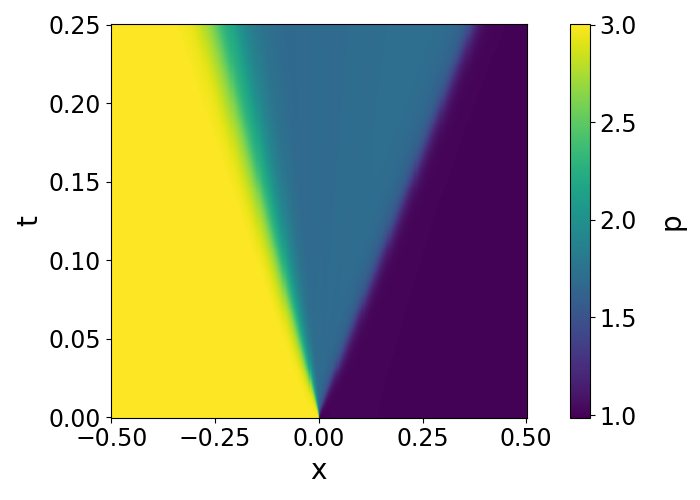}
    \caption{predicted $p$}\label{fig:sod-p}\end{subfigure}\hfill
  \begin{subfigure}[b]{0.32\textwidth}\centering
    \includegraphics[width=\textwidth]{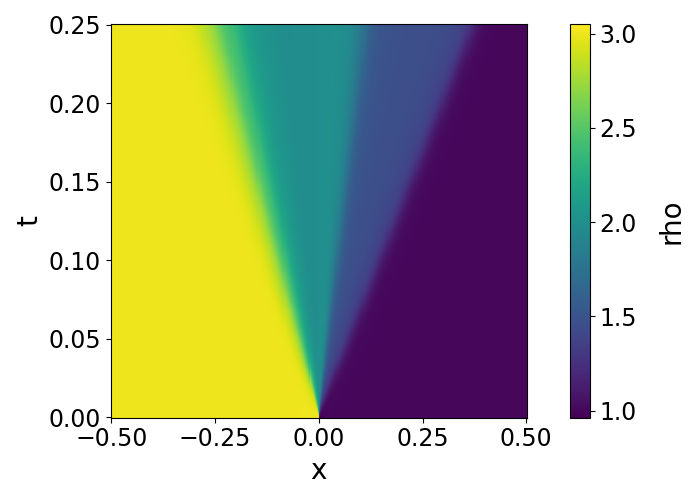}
    \caption{predicted $\rho$}\label{fig:sod-rho}\end{subfigure}\hfill
  \begin{subfigure}[b]{0.32\textwidth}\centering
    \includegraphics[width=\textwidth]{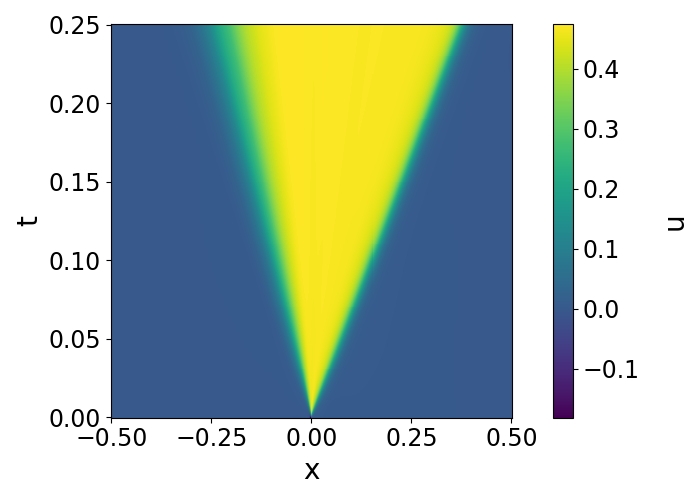}
    \caption{predicted $u$}\label{fig:sod-u}\end{subfigure}

  \vspace{0.5em}
  \begin{subfigure}[b]{0.32\textwidth}\centering
    \includegraphics[width=\textwidth]{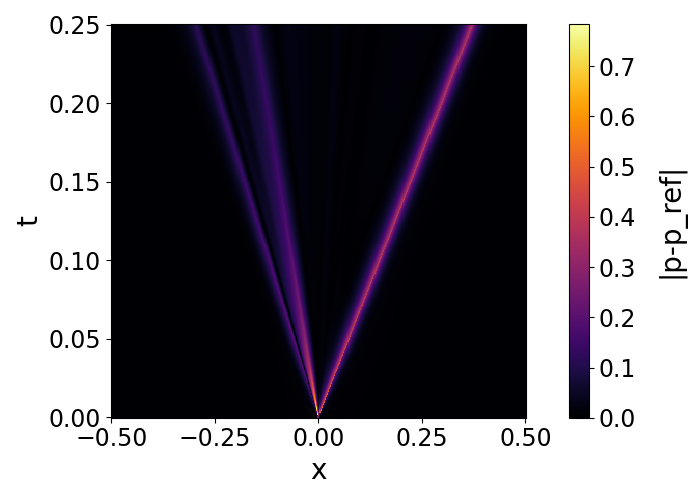}
    \caption{$p$, absolute error}\label{fig:sod-p-err}\end{subfigure}\hfill
  \begin{subfigure}[b]{0.32\textwidth}\centering
    \includegraphics[width=\textwidth]{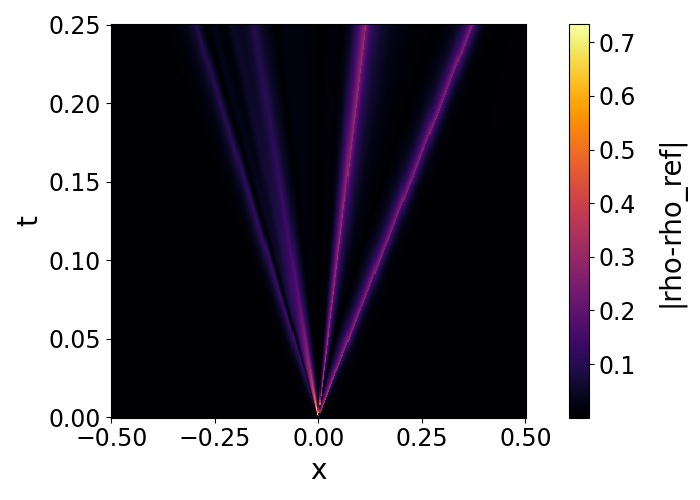}
    \caption{$\rho$, absolute error}\label{fig:sod-rho-err}\end{subfigure}\hfill
  \begin{subfigure}[b]{0.32\textwidth}\centering
    \includegraphics[width=\textwidth]{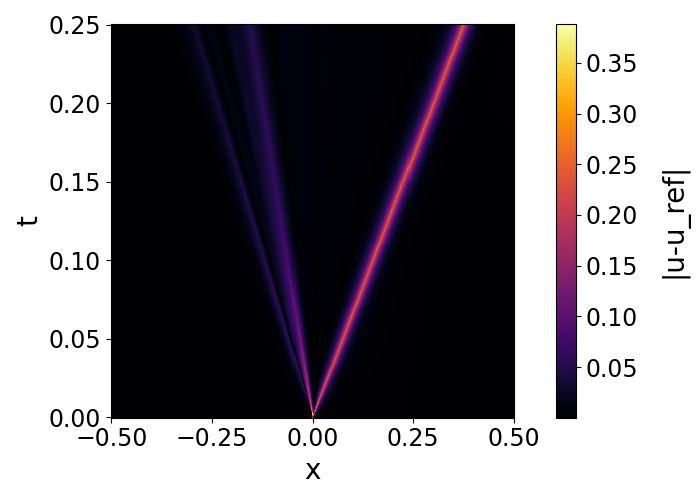}
    \caption{$u$, absolute error}\label{fig:sod-u-err}\end{subfigure}
  \caption{Sod shock tube: JXRGCM predictions and absolute errors.  Top panels
    show the predicted pressure $p$, density $\rho$, and velocity $u$; bottom
    panels show the absolute error against the exact Riemann solution
    \cite{Toro2009}.  The three wave families---the left rarefaction, the contact
    discontinuity, and the right shock---are resolved simultaneously, with error
    concentrated along the contact and the shock, where the density jump is
    sharpest.}
  \label{fig:sod}
\end{figure}

\begin{figure}[htbp]
    \centering
    \includegraphics[width=0.9\linewidth]{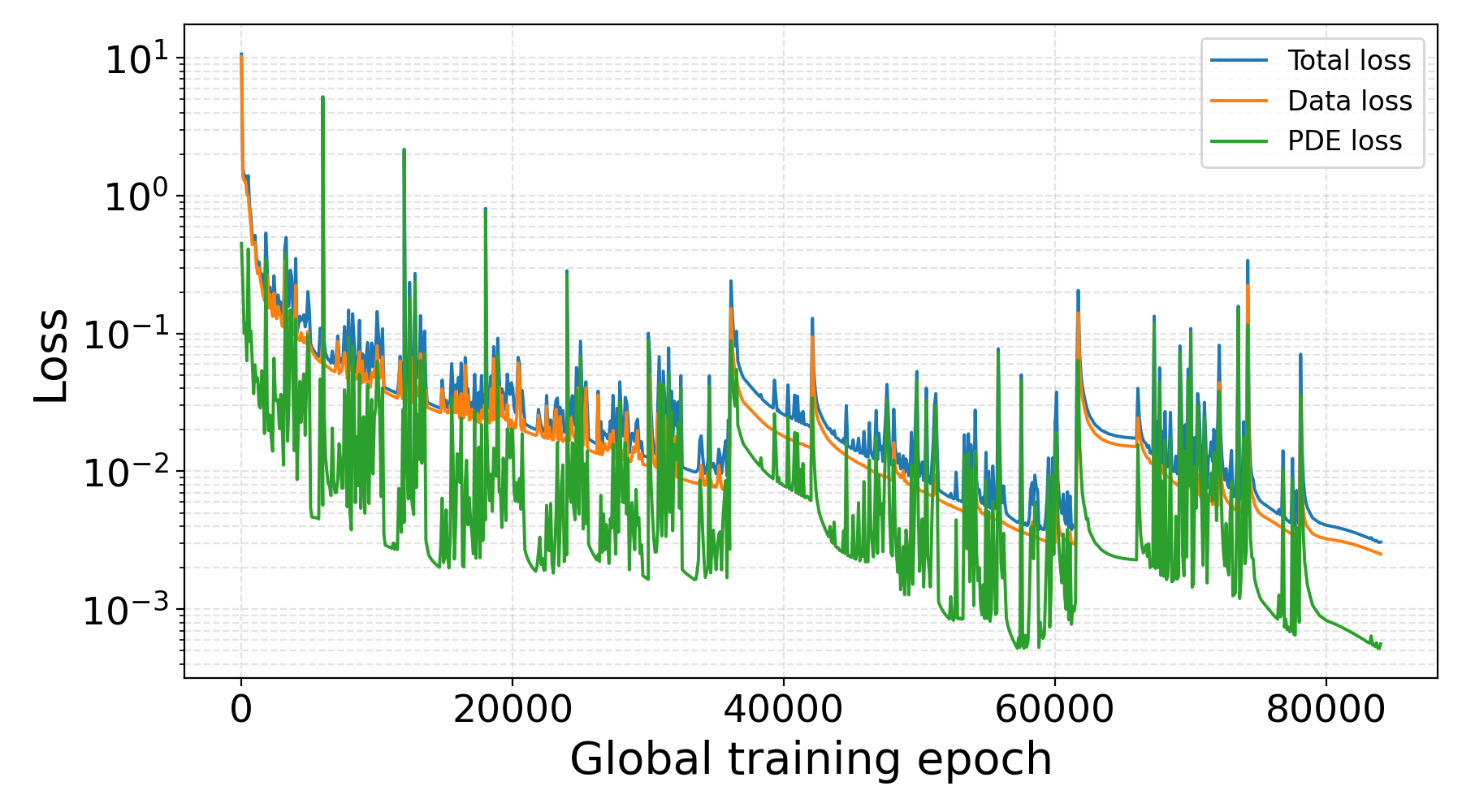}
    \caption{The convergence curves of the loss function for the Sod shock tube problem: The severe fluctuation in the PDE loss is mainly attributed to the heightened non-linearity of the equations and the identical residual weights adopted across numerical experiments for generalizability. Although gradient clipping and normalization help alleviate this behavior, we nevertheless obtain favorable convergence performance.}
    \label{fig:sod loss}
\end{figure}

\begin{table}[htbp]
  \centering
  \caption{Root mean square error (RMSE) for the Sod shock tube problem.  All
    methods use the same network architecture, the same collocation points, and
    the same optimization protocol, so the comparison isolates the effect of the
    formulation.  JXRGCM attains the smallest error in each of the three
    variables.}
  \label{tab:sod}
  \smallskip
  \begin{tabular}{lccc}
    \toprule
    Method & $\rho$ (RMSE) & $u$ (RMSE) & $p$ (RMSE) \\
    \midrule
    Artificial Viscosity & $2.13\times10^{-1}$ & $1.44\times10^{-1}$ & $2.53\times10^{-1}$ \\
    GA-PINN & $5.62\times10^{-2}$ & $4.12\times10^{-2}$ & $6.95\times10^{-2}$\\
    CI-PINN      & $1.18\times10^{-1}$ & $7.81\times10^{-2}$ & $1.11\times10^{-1}$ \\
    JXRGCM       & $4.60\times10^{-2}$ & $2.89\times10^{-2}$ & $4.85\times10^{-2}$ \\
    \bottomrule
  \end{tabular}
\end{table}

%======================================================================
\section{Conclusion}\label{sec:conclusion}
We introduced JXRGCM, a PINN framework that solves hyperbolic conservation laws by training on the Jin--Xin relaxation system and treating the relaxation parameter $\varepsilon$ as a continuation parameter driven to zero, addressing the conflict between network smoothness and limiting discontinuity present in fixed-$\varepsilon$ relaxation PINNs.  The analysis establishes regularity of the relaxation system (with a uniform-in-$\varepsilon$ existence time), an approximation-theoretic generalization bound, and an $\varepsilon$-uniform total-error estimate obtained through a sub-characteristic symmetrizer; for scalar conservation laws, combined with the relaxation limit, these results give an $L^2$-error bound of order $\mathcal{O}(\varepsilon^{1/4})$ for convergence of the numerical solution to the entropy solution, while a quadrature bound is introduced to establish a theoretical link between training and generalization errors.  We present three numerical examples: the Burgers equation, the shallow-water dam-break problem, and the Sod shock tube problem.  The scalar Burgers example is consistent with the scalar relaxation-limit setting, although its discontinuous initial data fall outside the smooth regularity assumptions used in the stability estimate; the shallow-water and Euler examples lie outside the convergence theory as well, since an explicit relaxation-limit rate for systems is not available (\cref{rem:scalar-scope}).  All three are reported as empirical evidence that the continuation strategy transfers across wave structures.

\section*{Acknowledgments}
Part of this work was done during the visit of Y.T. to the Department of Mathematics, University of California, Santa Barbara.  Y.T. would like to thank the department for the hospitality.  Y.T. is also grateful to Professor Weian Yong for useful discussions.

\appendix
\crefalias{section}{appendix}
\FloatBarrier
\section{Chapman--Enskog expansion}\label{app:CE}
We sketch the derivation of \cref{eq:chapman-enskog,eq:subchar} for the scalar
($2\times2$) system $\partial_t u+\partial_x v=0$,
$\varepsilon(\partial_t v+a^2\partial_x u)=f(u)-v$.  Expanding
$v=v_0+\varepsilon v_1+\mathcal{O}(\varepsilon^2)$ and matching orders gives
$v_0=f(u)$ and $v_1=-(\partial_t v_0+a^2\partial_x u)$.  Substituting into the
first equation,
\[
  \partial_t u+\partial_x f(u)
  = -\varepsilon\,\partial_x v_1
  = \varepsilon\,\partial_x\bigl(\partial_t f(u)+a^2\partial_x u\bigr)
  = \varepsilon\,\partial_x\!\Bigl(\bigl(a^2-f'(u)^2\bigr)\partial_x u\Bigr),
\]
using $\partial_t f(u)=f'(u)\partial_t u=-f'(u)\partial_x f(u)
=-f'(u)^2\partial_x u$ at leading order.  The correction is dissipative iff
$a^2-f'(u)^2>0$, i.e.\ $|f'(u)|<a$; the system version is \cref{eq:subchar}.

\FloatBarrier
\bibliographystyle{siamplain}
\bibliography{references}

\end{document}